\documentclass[a4paper,11pt]{amsart}
\usepackage{amsmath, amsfonts, amssymb, amsthm, amscd}
\usepackage[colorlinks=true, linkcolor=blue, citecolor=blue]{hyperref}
\usepackage{graphicx}
\usepackage{enumerate}
\usepackage{url}
\usepackage{color}
\usepackage[utf8]{inputenc}
\usepackage{tikz}
\usetikzlibrary{arrows} 
\usepackage[a4paper]{geometry}
\usepackage{dsfont}
\usepackage{mathrsfs}

\numberwithin{equation}{section}

\newtheorem{theorem}{Theorem}[section]
\newtheorem{lemma}[theorem]{Lemma}
\newtheorem{proposition}[theorem]{Proposition}
\newtheorem{corollary}[theorem]{Corollary}
\newtheorem{remark}[theorem]{Remark}

\newcommand{\bbN}{{\ensuremath{\mathbb N}} }

\newcommand{\bbR}{{\ensuremath{\mathbb R}} }

\newcommand{\bbZ}{{\ensuremath{\mathbb Z}} }

\newcommand{\cB}{{\ensuremath{\mathcal B}} }
\newcommand{\cC}{{\ensuremath{\mathcal C}} }
\newcommand{\cD}{{\ensuremath{\mathcal D}} }

\newcommand{\cF}{{\ensuremath{\mathcal F}} }

\newcommand{\cM}{{\ensuremath{\mathcal M}} }

\newcommand{\cU}{{\ensuremath{\mathcal U}} }

\newcommand{\cX}{{\ensuremath{\mathcal X}} }

\newcommand{\ga}{\alpha}
\newcommand{\gb}{\beta}
\newcommand{\gga}{\gamma}

\newcommand{\gd}{\delta}

\newcommand{\gep}{\varepsilon} 

\newcommand{\gt}{\theta}

\newcommand{\gl}{\lambda}
\newcommand{\gL}{\Lambda}

\newcommand{\gs}{\sigma}
\newcommand{\gS}{\Sigma}

\newcommand{\bfD}{\mathbf{D}}

\renewcommand{\tilde}{\widetilde}          
\DeclareMathSymbol{\leqslant}{\mathalpha}{AMSa}{"36} 
\DeclareMathSymbol{\geqslant}{\mathalpha}{AMSa}{"3E} 
\DeclareMathSymbol{\eset}{\mathalpha}{AMSb}{"3F}     
\newcommand{\dd}{\text{\rm d}}             
\newcommand{\suptwo}[2]{\sup_{\substack{#1 \\ #2}}} 
\newcommand{\inttwo}[2]{\int_{\substack{#1 \\ #2}}}     

\newcommand{\R}{\mathbb{R}}

\newcommand{\Z}{\mathbb{Z}}
\newcommand{\N}{\mathbb{N}}

\newcommand\bP{\ensuremath{\mathrm{P}}}
\newcommand\bE{\ensuremath{\mathrm{E}}}

\renewcommand{\epsilon}{\varepsilon}

\newcommand{\ind}{{\sf 1}}

\newenvironment{myenumerate}{
\renewcommand{\theenumi}{\arabic{enumi}}
\renewcommand{\labelenumi}{{\rm(\theenumi)}}
\begin{list}{\labelenumi}
{
\setlength{\itemsep}{0.4em}
\setlength{\topsep}{0.5em}
\setlength\leftmargin{2.45em}
\setlength\labelwidth{2.05em}
\setlength{\labelsep}{0.4em}
\usecounter{enumi}
}
}
{\end{list}
}

\renewenvironment{enumerate}{
\begin{myenumerate}}
{\end{myenumerate}}

\newcommand{\red}{\color{red}}

\newcommand{\beq}{\begin{equation}}
\newcommand{\eeq}{\end{equation}}
\newcommand{\ba}{\begin{aligned}}
\newcommand{\ea}{\end{aligned}}

\newcommand{\supp}{\mathrm{Supp}}

\newcommand{\Dtwo}{\mathbf{D}^{(2)}}

\begin{document}

\title[]{Strong Large Deviation principles for pair empirical measures of random walks in the Mukherjee-Varadhan topology}
\author{Dirk Erhard and Julien Poisat}
\date{today}

\begin{abstract}
In this paper we introduce a topology under which the pair empirical measure of a large class of random walks satisfies a strong Large Deviation principle. The definition of the topology is inspired by the recent article by Mukherjee and Varadhan~\cite{MV2016}. This topology is natural for translation-invariant problems such as the downward deviations of the volume of a Wiener sausage or simple random walk, known as the Swiss cheese model~\cite{BBH2001}. We also adapt our result to some rescaled random walks and provide a contraction principle to the single empirical measure despite a lack of continuity from the projection map, using the notion of diagonal tightness.
\end{abstract}

\thanks{\\
(D. Erhard) {\it Instituto de Matemática - Universidade Federal da Bahia}\\
(J. Poisat) {\it Université Paris-Dauphine, CNRS, UMR 7534, CEREMADE, PSL Research University, 75016 Paris, France}
}
 
\date{\today}

\maketitle

\tableofcontents

\section{Introduction}
\label{sec:intro}
Let $X=(X_n)_{n\in \bbN_0}$ be a Markov chain on a Polish space $\gS$ with transition kernel $\pi$ which we assume to have a density $p(x,y)$ with respect to a reference measure $\gl(\dd y)$. In the seventies, Donsker and Varadhan~\cite[(II)]{DV75} showed that the empirical measure defined by
\beq
\label{def:empr-meas}
L_n = \frac1n \sum_{i=1}^n \gd_{X_i}
\eeq
satisfies a Large Deviation principle in the space of probability measures equipped with the weak topology, when $\gS$ is \emph{compact} and $p(x,y)$ is uniformly bounded from above and below. If the state space $\gS$ is \emph{not compact} (e.g.\ $\gS=\bbR^d$ equipped with the Euclidian distance), the upper bound holds for \emph{compact} sets rather than \emph{closed} sets, while the lower bound still holds for any open set under the following assumption: for all $x\in\gS$, for all measurable sets $A\subset \gS$ such that $\gl(A)>0$, there exists $k\in\bbN$ such that $\pi^k(x,A)>~0$, see~\cite[Corollary 3.4 and Equation (4.1)]{DV76}. 
In some cases, this \emph{weak} Large Deviation principle may be upgraded to a standard (or \emph{strong}) one, recovering the Large Deviation upper bound for all closed sets: see e.g.\ the uniformity assumption in~\cite[Corollary 6.5.10]{DemboZei10:book}, the compactness conditions in~\cite{BFG15} for continuous-time Markov chains, or~\cite{JJW24} for a more general class of processes that encompasses both discrete-time and continuous-time Markov chains. However such cases fail to include many natural examples. In applications, the lack of compactness may be dealt with by adding a confining drift to the Markov chain (or diffusion)~\cite{MR709647} or folding it on a large torus~\cite{B94,MR1434118,DV79}. Quite recently, Mukherjee and Varadhan~\cite{MV2016} proposed a new approach in which they embed the space of probability measures on $\bbR^d$ into a larger space equipped with a certain topology that makes it a \emph{compact} metric space. This approach is actually very much in spirit of the concept of concentration-compactness introduced by P. L. Lions in the 80's, see~\cite{PL1985}, which cures the lack of compactness coming from the action of a noncompact group action like for instance rescaling or translation. Under this new topology, Mukherjee and Varadhan were able to prove a \emph{strong} Large Deviation principle for the empirical measure of Brownian motion~\cite[Theorem 4.1]{MV2016}, which was then successfully applied to the so-called polaron problem in~\cite{BKM2017, KM2017, MV2016}, and also in~\cite{MV20,MV20b} where the approach from~\cite{MV2016} has been developed for level-3 large deviations. Recently, the compactification of measures has also proven fruitful in the context of directed polymers, see~\cite{BatCha20,BroMuk19}. The reader may refer to~\cite{DemboZei10:book,dH-book2000,DS,Var-LD2016} for an account on Large Deviation theory and to~\cite{Var18} for the role of topology in this theory.

\par In this paper we adapt and extend the work of Mukherjee and Varadhan in order to prove a \emph{strong} Large Deviation principle (LDP) for the \emph{pair} empirical measure of the Markov chain $(X_n)_{n\in\N_0}$, defined by
\beq
\label{def:pair-empr-meas}
L_n^{(2)} = \frac1n \sum_{i=1}^n \gd_{(X_{i-1}, X_{i})},
\eeq
in the case $\gS=\bbR^d$. Our work is motivated by the application of this LDP to the so-called \emph{Swiss cheese} problem, that is the (downward) Large Deviations for the volume of a Wiener sausage in $\bbR^d$ ($d\ge 2$), by van den Berg, Bolthausen and den Hollander~\cite{BBH2001}. In that paper (as well as in Phetpradap's Ph.D thesis~\cite{Phetpradap} for the discrete random walk counterpart a few years later) the authors used the aforementioned folding procedure on a large torus to deal with the lack of compactness of the state space. In a companion paper~\cite{EP2023}, we apply the strong LDP for the pair empirical measure to the Swiss cheese problem in order to obtain the so-called \emph{tube property}.  The latter essentially means that, conditioned to the downward deviations of the random walk volume, the empirical measure of a certain  random walk skeleton introduced in~\cite{BBH2001} is close to one of the minimizers of the variational problem in the corresponding rate function.

\par The present paper is organized as follows. In Section~\ref{sec:topology} we introduce the relevant  notation and topology. Although there is a lot in common with~\cite{MV2016}, let us stress that this is \emph{not} the Mukherjee-Varadhan topology applied to the product space $\bbR^d\times \bbR^d$ instead of $\bbR^d$, see also Remark~\ref{rmk:not-the-same} below. Our main result,  that is the Large Deviation principle for the pair empirical measure of a large class of random walks, in this new topology, is stated in Theorem~\ref{thm:LDPpair} of Section~\ref{sec:LDP}. Lower semi-continuity of the rate function is proven in Section~\ref{sec:lsc}. The lower and upper bounds of the LDP are proven in Sections~\ref{sec:lb} and~\ref{sec:ub}, respectively.  We shall use therein the well-known fact that $(X_n,X_{n+1})_{n\in\bbN_0}$ is itself a Markov chain. In Section~\ref{sec:adapt} we adapt our result to the case of certain rescaled random walks. Finally, in Section~\ref{sec:contraction} we explain how to use~Theorem~\ref{thm:LDPpair} to obtain a Large Deviation principle for the (single) empirical measure. This part is non-trivial since, due to a lack of continuity from the projection map, the standard contraction principle is not applicable. We use the new notion of {\it diagonal tightness} to overcome this difficulty. In order to make the paper more self-contained, we include a condensed presentation of the Mukherjee-Varadhan topology in Appendix~\ref{sec:a}.
\section{Topology on the space of probability measures modulo shifts}
\label{sec:topology}
Throughout the paper, we equip $\bbR^d$ with any of its equivalent norms, further denoted by $|\cdot|$, while $\bbR^d\times \bbR^d$ is equipped with the product norm $||(x,y)|| = |x|\vee |y| = \max(|x|, |y|)$.
Let $\cM_1^{(2)} := \cM_1(\bbR^d\times \bbR^d)$ be the space of probability measures on $\bbR^d\times \bbR^d$ and $\cM_{\le1}^{(2)} := \cM_{\le1}(\bbR^d\times \bbR^d)$ be the space of sub-probability measures on $\bbR^d\times \bbR^d$. We consider the action of the \emph{diagonal} shifts $\gt_{x,x}$ for $x\in\bbR^d$ , defined by:
\beq
\int_{\bbR^d\times \bbR^d} f(u,v) (\gt_{x,x} \nu)(\dd u,\dd v) = \int_{\bbR^d\times \bbR^d} f(u+x,v+x)  \nu(\dd u,\dd v)
\eeq
for all continuous bounded functions $f\colon \bbR^d\times \bbR^d \mapsto \bbR$ and $\nu\in \cM_{\le1}^{(2)}$. We shall denote by $\tilde \cM_1^{(2)}$ (resp. $\tilde \cM_{\le 1}^{(2)}$) the space of equivalence classes of $\cM_1^{(2)}$ (resp. $\cM_{\le 1}^{(2)}$) under the collection of shifts $\gt_{x,x}$. For $k\ge 2$, we define $\cF_k$ as the space of continuous functions $f\colon (\bbR^d)^k \mapsto \bbR$ that are translation invariant, i.e.
\beq
f(x_1+x, \ldots, x_k +x) = f(x_1,\ldots, x_k), \quad \forall x,x_1,\ldots, x_k \in \bbR^d,
\eeq
and {\it vanishing at infinity}, in the sense that
\beq
\lim_{\max_{i\neq j} |x_i-x_j|\to \infty} f(x_1, \ldots, x_k) = 0.
\eeq
 Let us define $\cF_k^{(2)}= \cF_{2k}$ for $k\ge 1$. When $f\in \cF_k^{(2)}$ and $\ga \in \cM_{\le 1}^{(2)}$, we write
\beq
\label{eq:defLambda}
\gL(f, \ga) := \int f(u_1, v_1,\ldots, u_{k}, v_{k}) \prod_{1\le i\le k} \ga(\dd u_i, \dd v_i),
\eeq
which only depends on the orbit $\tilde \ga$.
\begin{remark}
\label{rmk:not-the-same}
Note that the space  $\tilde \cM_1^{(2)}$ defined here is different from $\tilde \cM_1(\R^{2d})$ defined in \cite{MV2016}. Indeed, in~\cite{MV2016} the shifts are with respect to all directions in $\R^{2d}$ whereas here they are only with respect to all directions of the form $(x,x)\in \R^{2d}$. This is very natural in view of our application to the pair empirical measure.
\end{remark}
\begin{remark}
\label{rmk:not-the-same2}
There is another natural choice for the set of test functions, which would lead to the definition of an a priori \emph{stronger} topology. We could say that $f$ (with an even number of arguments) is \emph{vanishing at infinity} if
\beq
\lim_{\max_{i\neq j} \|(u_i,v_i)-(u_j,v_j)\|\to \infty} f(u_1, v_1, \ldots, u_k, v_k) = 0.
\eeq
Note that this is a larger set of test functions. Indeed, if we write $(x_1, x_2, \ldots, x_{2k})=(u_1, v_1, \ldots, u_k, v_k)$ then
\beq
\max_{i\neq j} |x_i-x_j| \ge
\max_{i\neq j} (|u_i-u_j| \vee |v_i-v_j|)= \max_{i\neq j} \|(u_i,v_i)-(u_j,v_j)\|.
\eeq%
The test function we exhibit in Remark~\ref{rmk:not-the-same3} below shows that the inclusion is actually strict.
\end{remark}
\subsection{Vague and weak convergence}
Let us recall that a sequence of sub-probability measures $(\mu_n)_{n\in \N_0}$ on $\bbR^d$ is said to converge vaguely (resp.\ weakly) to $\mu$ if for every continuous and compactly supported (resp.\ continuous and bounded) function $f\colon \bbR^d \to \bbR$,
\beq
\lim_{n\to \infty} \int f(x) \mu_n(\dd x) = \int f(x) \mu(\dd x).
\eeq
Weak convergence clearly implies vague convergence but the reverse implication is false~\cite[Chapter 5, Section 28]{BillingBookPM}. However, the following lemma will be useful for later purposes:%
\begin{lemma}[\bf See Lemma 2.2 in~\cite{MV2016}]
\label{lem:vagueweak}
Consider a sequence $(\mu_n)_{n\in \N_0}$ of sub-probability measures in $\R^d$ that converges vaguely to some sub-probability measure $\alpha$. Then we can write $\mu_n= \alpha_n + \beta_n$ for every $n\in\N_0$, where $(\alpha_n)_{n\in\N_0}$ converges weakly to $\alpha$ and $(\beta_n)_{n\in\N_0}$ converges vaguely to zero. Moreover, for any $n\in\N$ there exists a positive number $R_n$ such that for any measurable set $A$ one has $\alpha_n(A)=\mu_n(A\cap B(0,R_n))$ and $\beta_n(A)=\mu_n(A\cap B(0,R_n)^c)$, where $B(0,R_n)$ is the closed ball in $\bbR^d$ centered at the origin with radius $R_n$. Furthermore, the sequence $(R_n)_{n\in\N}$ can be chosen to be increasing and tending to infinity.
\end{lemma}
The last sentence is the only part that was not explicitly mentioned in~\cite[Lemma 2.2]{MV2016} but it is derived in the proof given therein. We recall that the measure $\alpha_n$ is defined in~\cite{MV2016} as the restriction of $\mu_n$ to $B(0,R_n)$ and $\beta_n$ as the restriction of $\mu_n$ to $B(0,R_n)^c$, where $(R_n)_{n\in\N_0}$ is an appropriately chosen increasing sequence which converges to infinity and satisfies
\begin{equation}
	\mu_n(B(0,R_n)) \leq \alpha(\R^d) +\frac{1}{R_n}\,.
\end{equation}
\subsection{Widely separated sequences} We say that two sequences $(\ga_n)$ and $(\gb_n)$ in $\cM_{\le 1}^{(2)}$ are {\it widely separated} if, for some positive function $V$ in $\cF_2^{(2)}$,
\beq
\label{eq:wide-sep}
\lim_{n\to \infty} \int_{(\bbR^d)^4} V(u_1, v_1, u_2, v_2) \ga_n(\dd u_1, \dd v_1) \gb_n(\dd u_2, \dd v_2) = 0.
\eeq

\begin{remark}
\label{rmk:wide-sep}
It is clear from~\eqref{eq:wide-sep} that if $(\ga_n)$ and $(\gb_n)$ are widely separated and if $\gb_n = \gb_n^{(1)}+\gb_n^{(2)}$, where $\gb_n^{(1)}$ and $\gb_n^{(2)}$ are sub-probability measures, then also $(\ga_n)$ and $(\gb_n^{(i)})$ are widely separated, for every $i\in\{1,2\}$.
\end{remark}

\begin{remark}
\label{rmk:not-the-same3} There exist $(\ga_n)$ and $(\gb_n)$ that are widely separated for our choice of test functions but not widely separated for the larger set of test functions mentioned in Remark~\ref{rmk:not-the-same2}. Indeed, let $\ga_n$ be a Dirac mass at $(u(n), v(n))\in \bbR^d\times \bbR^d$, where $|u(n)-v(n)|\to \infty$. It is clear that $(\ga_n)$ is widely separated (in our sense) from any other sequence (including itself). Now, consider the bounded continuous function
\beq
V(u_1, v_1, u_2, v_2) = \frac{1}{(1+ |u_1- u_2|)(1+ |v_1- v_2|)}.
\eeq
The function $V$ is not vanishing at infinity in our sense, since $V(u_1, v_1, u_2, v_2) = V(u_1+x, v_1, u_2+x, v_2)$ for every $x\in\bbR^d$, but it is vanishing at infinity in the less restrictive sense of Remark~\ref{rmk:not-the-same2} and
\beq
\int V(u_1, v_1, u_2, v_2) \ga_n(\dd u_1, \dd v_1) \ga_n(\dd u_2, \dd v_2) = 1
\eeq
does not converge to zero.
\end{remark}
The following lemma, which mimicks~\cite[Lemma 2.4]{MV2016}, lists the most important properties of widely separated sequences of sub-probability measures.
\begin{lemma}\label{lem:widely-sep}
Let $(\alpha_n)$ and $(\beta_n)$ be two widely separated sequences in $\cM_{\le 1}^{(2)}$. Then, 
\begin{enumerate}
	\item For every $W\in \cF_2^{(2)}$,
	\beq
	\lim_{n\to\infty} \int W(u_1, v_1, u_2, v_2) \alpha_n(\dd u_1, \dd v_1)\, \beta_n(\dd u_2, \dd v_2) = 0\,.
	\eeq
	\item For every $k\geq 2$ and every $f\in \cF_k^{(2)}$,
	\beq
	 \lim_{n\to \infty} |\gL(f, \ga_n + \gb_n)-\gL(f, \ga_n)-\gL(f,\gb_n)| = 0.
	\eeq	
\end{enumerate}
\end{lemma}
\begin{proof}[Proof of Lemma~\ref{lem:widely-sep}]
Let $W\in \cF_2^{(2)}$ be arbitrary and let $V\in\cF_2^{(2)}$ be the positive function from~\eqref{eq:wide-sep}. Then, since $V(0,v_1, u_2, v_2)$ is bounded from below by a positive constant on compact sets, for every $\gep >0$ there exists a constant $C_\gep$ such that for any $v_1, u_2, v_2\in\bbR^d$,
\beq
|W(0,v_1, u_2, v_2)| \leq C_\gep V(0,v_1, u_2, v_2)+\gep\,.
\eeq
Thus
\beq
\ba
\limsup_{n\to \infty} &\int |W(u_1, v_1, u_2, v_2)| \alpha_n(\dd u_1, \dd v_1)\beta_n(\dd u_2, \dd v_2)\\
&\leq C_\gep\limsup_{n\to\infty} \int V(0, v_1-u_1, u_2-u_1, v_2-u_1) \alpha_n(\dd u_1, \dd v_1)\beta_n(\dd u_2, \dd v_2) + \gep\\
&=\gep\,,
\ea
\eeq
and (1) follows.
To see that (2) holds as well, we first note that the case $k=2$ is a direct consequence of the first part of the lemma. The case $k\ge 3$ follows easily: w.l.o.g, any cross-term in the expansion of $\gL(f,\ga_n+\gb_n)$ may be written as
\beq
\int f(u_1,v_1, u_2, v_2, \ldots, u_k, v_k) \ga_n(u_1,v_1) \gb_n(u_2,v_2) \prod_{3\le i\le k} \gga_{n,i}(u_i,v_i),
\eeq
where $\gga_{n,i}$ is either $\ga_n$ or $\gb_n$. Using translation invariance and repeating the argument used at the beginning of the proof, we see that for every $\gep>0$, there exists $C_\gep>0$ such that for all $v_1, u_2, v_2, \ldots, u_k, v_k\in \bbR^d$,
\beq
|f(0,v_1, u_2, v_2, \ldots, u_k, v_k)| \le C_\gep V(0,v_1,u_2,v_2) + \gep,
\eeq
which allows to conclude.
\end{proof}

Here is a sufficient condition for two sequences of measures to be widely separated.
\begin{lemma}\label{lem:wid-sep-suff}
Let $(\ga_n)$ and $(\gb_n)$ be two sequences in $\cM_{\le 1}^{(2)}$. If $(\ga_n)$ is tight and $(\gb_n)$ converges vaguely to zero, then they are widely separated.
\end{lemma}
\begin{proof}[Proof of Lemma~\ref{lem:wid-sep-suff}]
Let $V\in \cF_2^{(2)}$ and $\gep>0$. By tightness of $(\ga_n)$ and boundedness of $V$, there exists $M>0$ such that  for all $n\ge 1$,
\beq
\int V(u_1,v_1,u_2, v_2) \ga_n(\dd u_1, \dd v_1) \gb_n(\dd u_2, \dd v_2) \le \int_{|u_1|, |v_1|\le M} (\ldots) + \gep \|V\|_{\infty}.
\eeq
Here, $(\ldots)$ stands of course for $V(u_1,v_1,u_2, v_2) \ga_n(\dd u_1, \dd v_1) \gb_n(\dd u_2, \dd v_2)$.
We further split the integral on the right-hand side as
\beq
\inttwo{|u_1|, |v_1|\le M}{|u_2|,|v_2|\le 2M} (\ldots) + \inttwo{|u_1|, |v_1|\le M}{ |u_2|\vee|v_2|> 2M} (\ldots)\ .
\eeq
We claim that we can make the second term smaller than $\gep$ by choosing $M$ even larger if necessary. Indeed, since $|u_2-u_1|\vee |v_2-v_1|\ge M$ on the domain of integration and $V$ is vanishing at infinity the claim follows. As for the first term, it goes to zero as $n$ goes to infinity since $V$ is bounded and $(\gb_n)$ converges vaguely to zero.
\end{proof}
We close this section with a property that was used without proof in the original paper~\cite{MV2016} and that will be useful in the sequel. We provide a proof here:
\begin{lemma}
\label{lem:wid-sep-div-seq}
Let $(\ga_{n}^{(1)})$ and $(\ga_{n}^{(2)})$ be two sequences in $\cM_{\le 1}^{(2)}$ that are widely separated.
If there exist two $\bbR^d$-valued sequences $(a_{n}^{(1)})$ and $(a_{n}^{(2)})$ such that $\ga_{n}^{(i)} \star \gd_{(a_{n}^{(i)},a_{n}^{(i)})}$ converges weakly to some \emph{non-zero} $\ga_i\in\cM_{\le 1}$ for every $i\in\{1,2\}$, then $|a_{n}^{(1)}- a_{n}^{(2)}|$ diverges.
\end{lemma}
\begin{proof}[Proof of Lemma~\ref{lem:wid-sep-div-seq}]
For ease of notation, write $b_n := a_{n}^{(1)}-a_{n}^{(2)}$ and $\hat \ga_{n}^{(i)} := \ga_{n}^{(i)} \star \gd_{(a_{n}^{(i)},a_{n}^{(i)})}$ for every $i\in\{1,2\}$. Let $W$ be a \emph{positive} function in $\cF_2^{(2)}$. By translation invariance of $W$, we have
\beq
\label{eq:wid-sep-div-seq}
\ba
&\int W(u_1, v_1, u_2, v_2) \ga_{n}^{(1)}(\dd u_1, \dd v_1) \ga_{n}^{(2)}(\dd u_2, \dd v_2)\\
\qquad &=
\int W(u_1, v_1, u_2 + b_n, v_2+b_n) \hat\ga_{n}^{(1)}(\dd u_1, \dd v_1) \hat\ga_{n}^{(2)}(\dd u_2, \dd v_2).
\ea
\eeq
Assume, by contradiction, that $(b_n)$ does \emph{not} diverge. Then, there exists $R>0$ such that $|b_n|\le R$ along a subsequence. Additionally, we may assume (possibly enlarging $R$) that
\beq
\ga_i(B((0,0),R)) >0, \qquad i\in\{1,2\},
\eeq
where $B((0,0),R)$ is the closed ball of $\bbR^d \times \bbR^d$ equipped with the product norm $\|\cdot\|$ defined at the beginning of Section~\ref{sec:topology}.
Restricting the integral in the r.h.s.\ of~\eqref{eq:wid-sep-div-seq} to $|u_1|\vee |u_2| \vee |v_1|\vee |v_2|\le R$, we obtain (along a subsequence)
\beq
\liminf_{n\to \infty} {\rm \eqref{eq:wid-sep-div-seq}} \ge \Big(\inf_{B((0,0),R)\times B((0,0),2R)} W \Big)\times \prod_{i\in\{1,2\}}\ga_i(B((0,0),R)) >0,
\eeq
which contradicts the fact that $(\ga_{n}^{(1)})$ and $(\ga_{n}^{(2)})$ are widely separated.
\end{proof}
\subsection{Totally disintegrating sequences} From now on and unless stated otherwise, we denote, for every $(x,y)\in \bbR^d\times \bbR^d$ and $r>0$,
\beq
B((x,y), r) := \{(u,v)\in \bbR^d\times \bbR^d \colon |u-x| \vee |v-y| < r\}.
\eeq
We say that a sequence $(\mu_n)$ in $\cM_{\le 1}^{(2)}$ is {\it totally disintegrating} if, for every $r>0$,
\beq
\label{eq:total-dis}
\lim_{n\to\infty} \sup_{x\in\bbR^d} \mu_n(B((x,x), r)) = 0.
\eeq
Following~\cite[Lemma 2.3]{MV2016}, we obtain the following result. 
\begin{lemma}\label{lem:total-dis}
The sequence $(\mu_n)$ in $\cM_{\le 1}^{(2)}$ is totally disintegrating iff one of the following equivalent statements holds:
\begin{enumerate}
	\item There exists a positive $V\in\cF_2^{(2)}$ such that $\lim_{n\to\infty} \gL(V,\mu_n)= 0$.
	\item For any $V\in\cF_2^{(2)}$, $\lim_{n\to\infty} \sup_{x,y\in\bbR^d} \int V(x,y,u,v) \mu_n(\dd u ,\dd v)=0$.
	\item For any $V\in\cF_2^{(2)}$, $\lim_{n\to\infty} \gL(V,\mu_n)= 0$.
\end{enumerate}
\end{lemma}
\begin{proof}[Proof of Lemma~\ref{lem:total-dis}] We prove that (1) $\Rightarrow$ \eqref{eq:total-dis} $\Rightarrow$ (2) $\Rightarrow$ (3). Clearly, (3) implies (1).
\par {\bf (i)} Let us prove that (1) implies~\eqref{eq:total-dis}. Letting
\beq
\gd := \min_{a,b,c\in B(0,2r)} V(0,a,b,c)>0,
\eeq
we get
\beq
\ba
&\int V(u_1, v_1, u_2, v_2) \mu_n(\dd u_1, \dd v_1)\mu_n(\dd u_2, \dd v_2)\\ 
&=\int V(0, v_1-u_1, u_2-u_1, v_2-u_1) \mu_n(\dd u_1, \dd v_1)\mu_n(\dd u_2, \dd v_2)\\
&\ge \gd \int_{v_1, u_2, v_2 \in B(u_1,2r )} \mu_n(\dd u_1, \dd v_1)\mu_n(\dd u_2, \dd v_2)\\
&\ge \gd \sup_{x\in\bbR^d} \mu_n(B((x,x), r))^2.
\ea
\eeq

\par {\bf (ii)} From~\eqref{eq:total-dis} to (2). Let $x,y\in\bbR^d$.
Fix $M>0$ and write
\beq
\ba
&\int|V(x,y,u,v)| \mu_n(\dd u, \dd v)\\
&= \int_{B((x,x),M)} |V(x,y,u,v)| \mu_n(\dd u, \dd v) + \int_{B((x,x),M)^c}|V(x,y,u,v)| \mu_n(\dd u, \dd v).
\ea
\eeq
The first term on the right-hand side goes to zero uniformly in $(x,y)$ as $n$ tends to infinity by the boundedness of $V$ and by \eqref{eq:total-dis}, while the second term goes to zero uniformly in $(x,y)$ as $M$ tends to infinity by the fact that $V$ vanishes at infinity.

\par {\bf (iii)} To go from (2) to (3), it is enough to write
\beq
\gL(V,\mu_n) \le \sup_{x,y\in\bbR^d} \int V(x,y,u,v) \mu_n(\dd u, \dd v)
\eeq
using the fact that $\mu_n$ is a sub-probability measure.
\end{proof}

As an immediate corollary we obtain the following:

\begin{corollary}\label{cor:total-dis}
	If the sequence $(\mu_n)$ is totally disintegrating then, for any $k\geq 2$ and any $V\in \cF_k^{(2)}$,
	\beq
	\lim_{n\to \infty}\Lambda(V,\mu_n)=0\,.
	\eeq
\end{corollary}

\begin{proof}[Proof of Corollary~\ref{cor:total-dis}]
Apply Item (3) in Lemma~\ref{lem:total-dis} to 
\beq
W(u_1,v_1,u_2, v_2) := \sup_{u_3,v_3,\ldots,u_k, v_k\in\bbR^d} |V(u_1,v_1, \ldots, u_k, v_k)|.
\eeq
\end{proof}

\subsection{Compactification}
Let us define
\beq
\cF^{(2)} := \bigcup_{k\ge 2} \cF_k^{(2)},
\eeq
for which there exists a countable dense set (under the uniform metric) denoted by
\beq
\{f_r(u_1, v_1,\ldots, u_{k_r}, v_{k_r})\colon r\in \bbN\},
\eeq
(same arguments as in~\cite[Section 2.2]{MV2016}). With that we mean that for each $k\geq 2$ the family 
\beq
\cF_k^{(2)}\cap \{f_r(u_1, v_1,\ldots, u_{k_r}, v_{k_r})\colon r\in \bbN\}
\eeq is dense in $\cF_k^{(2)}$ with respect to the uniform metric. We define
\beq
\tilde\cX^{(2)} := \Big\{
\xi = \{\tilde \ga_i\}_{i\in I} \colon \ga_i \in \cM_{\le 1}^{(2)},\ \sum_{i\in I} \ga_i(\bbR^d \times \bbR^d) \le 1\Big\},
\eeq
where $I$ may be empty, finite or countably infinite. Recall~\eqref{eq:defLambda} and for any $\xi_1, \xi_2\in \tilde \cX^{(2)}$, define
\beq
\Dtwo(\xi_1, \xi_2) := \sum_{r\ge 1} \frac{1}{2^r} \frac{1}{1+\|f_r\|_{\infty}}
\Big|%
\sum_{\tilde\ga\in\xi_1} \gL(f_r, \ga)- \sum_{\tilde\ga\in\xi_2} \gL(f_r, \ga)
\Big|.%
\eeq

\begin{proposition}\label{lem:D2metric}
$\Dtwo$ is a metric on $\tilde\cX^{(2)}$.
\end{proposition}
\begin{proof}[Proof of Proposition~\ref{lem:D2metric}]
 As in~\cite[Theorem 3.1]{MV2016}, the only non-trivial part is to show that $\Dtwo(\xi_1,\xi_2)=0$ implies $\xi_1=\xi_2$. We follow the three-step proof from that reference.
\par {\bf Step 1.} We give some level of details to show how the proof is adapted to our case. If $\Dtwo(\xi_1,\xi_2)=0$ then for all $k\ge 2$ and $f\in\cF_k^{(2)}$,
\beq
\sum_{\tilde\ga\in\xi_1} \gL(f, \ga) = \sum_{\tilde\ga\in\xi_2} \gL(f, \ga).
\eeq
We deduce therefore that for every integer $r\ge 1$,
\beq
\sum_{\tilde\ga\in\xi_1} \gL(f, \ga)^r = \sum_{\tilde\ga\in\xi_2} \gL(f, \ga)^r.
\eeq
Indeed, define the function (for $r=2$)
\beq
\ba
&g_N(u_1, v_1, \ldots, u_k, v_k, u_{k+1}, v_{k+1}, \ldots, u_{2k}, v_{2k})\\
&:=
f(u_1, v_1, \ldots, u_k, v_k) f(u_{k+1}, v_{k+1}, \ldots, u_{2k}, v_{2k})
\varphi\Big(\frac{u_{k+1}-u_1}{N} \Big),
\ea
\eeq
where $0\le \varphi \le 1$ is equal to $1$ inside a ball of radius $1$ and is truncated smoothly to $0$ outside a ball of radius $2$. Then, $g_N\in \cF_{4k}$ and converges pointwise to $f\otimes f$ as $N\to \infty$. Hence, using the fact that $f$ is bounded and dominated convergence
\begin{equation}
\lim_{N\to\infty}	\Lambda(g_N,\alpha) = \Lambda(f,\alpha)^2\,.
\end{equation} 
The general case follows the same idea.
\par {\bf Step 2.} From there we prove that to every orbit $\tilde \ga_1$ in $\xi_1$ we can match a single orbit $\tilde \ga_2$ in $\xi_2$ such that $\gL(f, \ga_1) = \gL(f, \ga_2)$ for all $f\in \cF^{(2)}_k$ and $k\ge 2$. This part of the argument is exactly the same as in~\cite{MV2016}. We do not reproduce it here, for conciseness.
\par {\bf Step 3.} This step is also an adaptation of~\cite{MV2016}, so we only sketch the arguments. We want to recover the orbit of $\ga \in \cM_{\le 1}^{(2)}$ from the value of $\gL(f,\ga)$ for $f\in \cF_k^{(2)}$. Adapting~\cite{MV2016}, from these values we get those of
\beq
\prod_{j=1}^k \phi(s_j,t_j), \quad \text{where } \quad \phi(s,t) := \int e^{i\langle (s,t),(u,v) \rangle} \ga(\dd u, \dd v), \qquad s,t\in\bbR^d,
\eeq 
provided $\sum (s_j + t_j) = 0$. Suppose now that
\beq
\forall k\ge 1,\quad \prod_{j=1}^k \phi(s_j,t_j) = \prod_{j=1}^k \psi(s_j,t_j), \qquad
\text{whenever } \sum (s_j + t_j) = 0,
\eeq
 where $\psi$ is another characteristic function.
Following~\cite{MV2016}, we obtain that $|\phi(s,t)|= |\psi(s,t)|$ and write $\phi(s,t) = \psi(s,t)\chi(s,t)$ whenever $|\phi(s,t)|= |\psi(s,t)|\neq 0$. As soon as $\sum (s_j + t_j) = \gs +\tau \in \bbR^d$, we have $\prod_{j=1}^k \chi(s_j,t_j) = \chi(\gs,\tau)$, provided that the $s_j$'s and $t_j$'s are such that $|\psi(s_j,t_j)| = |\phi(s_j,t_j)|\neq 0$ for all $j$, and $\chi(\gs,\tau)\neq 0$. In particular,
\beq
\chi(s_1+s_2, t_1+t_2) = \chi(s_1,t_1)\chi(s_2, t_2),
\eeq
hence as in~\cite{MV2016} we can show that $\chi(s,t) = e^{i(\langle a_1, s \rangle + \langle a_2, t \rangle)}$ for some $a_1, a_2 \in \bbR^d$. The fact that actually $\chi(s_1+s_2+C, t_1+t_2-C) = \chi(s_1,t_1)\chi(s_2, t_2)$ for all $C\in\bbR^d$ entails that $a_1= a_2$. This means that $\ga$ is determined up to shifts by some $(a,a)\in \bbR^{2d}$, which ends the proof.
\end{proof}

\begin{proposition}\label{pr:D2compact}
The space $\tilde \cX^{(2)}$ equipped with $\Dtwo$ is a compact metric space that contains $\tilde\cM_1^{(2)}$ as a dense subset.
\end{proposition}

\begin{proof}[Proof of Proposition~\ref{pr:D2compact}]
This proof is an adaptation of the arguments used in ~\cite{MV2016}, so we only sketch the arguments.

{\bf Step 1.} Let us first show that $\tilde\cM_1^{(2)}$ is dense in $\tilde \cX^{(2)}$. Let $\xi = \{\tilde \ga_i, i\in I\} \in \tilde \cX^{(2)}$ and $\gep>0$. Pick a finite collection $\{\tilde \ga_i,\ 1\le i\le k\}$ such that $\sum_{i>k} p_i < \gep$, where $p_i$ denotes the total mass of $\tilde \ga_i$. Let $\ga_i$ be an arbitrary member of the orbit $\tilde \ga_i$ and for $M>0$, let $\gl_M$ be the Gaussian law on $\bbR^{2d}$ with zero mean and covariance matrix $M\times \rm{Id}$. Pick any $M$-parametrized sequence $(a_{i,M})_{1\le i \le k}$ in $(\bbR^d)^k$ such that $\inf_{i\neq j} |a_{i,M} - a_{j,M}| \to \infty$ as $M\to \infty$. Finally, set 
\beq
\label{eq:explicitmu}
\mu_M := \sum_{1\le i\le k} \ga_i * \gd_{(a_{i,M},a_{i,M})} + \Big(1 - \sum_{1\le i\le k} p_i \Big) \gl_M \in \cM_1^{(2)}\,.
\eeq 
The mutual distances of the centers of masses of the measures in the first sum increase to infinity, i.e., they are widely separated, see~\eqref{eq:wide-sep} and all the mass of $\gl_M$ vanishes in the limit as $M$ tends to infinity, i.e, $\gl_M$ is totally disintegrating, see~\eqref{eq:total-dis}. As a consequence, we get by Lemmas~\ref{lem:widely-sep} and \ref{lem:total-dis} that for all $k\ge 1$ and $f\in \cF_k^{(2)}$,
\beq
\lim_{M\to \infty} \Lambda(f, \mu_M)=\sum_{1\le i\le k} \Lambda(f, \tilde \ga_i).
\eeq
Since $\gep>0$ may be chosen arbitrarily small this completes the first step.

{\bf Step 2.} Let us show that for any sequence $(\mu_n)$ in $\cM_1^{(2)}$, there exists a subsequence along which $(\tilde \mu_n)$ converges to some element of $\tilde\cX^{(2)}$. Together with the first step this implies the result.

\par {\bf (i)} We start with some preliminary considerations. We use the following {\it concentration function}:
\beq
\label{eq:conc-fct}
q_{\mu}(r) := \sup_{x\in \bbR^d} \mu(B((x,x), r)), \qquad r\ge 0,\ \mu\in \cM_1^{(2)}.
\eeq
Let now $(\mu_n)_{n\in \N}$ be a sequence in $\cM_{\leq 1}^{(2)}$. Going over to a subsequence if necessary, we may define, by Helly's selection theorem,
\begin{equation}
	q:=\lim_{r\to \infty}\lim_{n\to\infty}q_{\mu_n}(r)\quad\text{and}\quad p:=\lim_{n\to\infty}\mu_n(\R^{2d})\,.
\end{equation}
If $q=0$, then it follows by Corollary~\ref{cor:total-dis} that $\tilde\mu_n\to 0$ in $\tilde \cX^{(2)}$. 
If on the other hand $q>0$, then taking a suitable translation vector $(a_n,a_n)\in \R^d\times\R^d$ we have that for some $r>0$ and all sufficiently large $n$, the shifted measure $\lambda_n= \mu_n*\delta_{(a_n,a_n)}$ satisfies
\begin{equation}\label{eq:q2}
	\lambda_n(B((0,0),r))\geq q/2\,.
\end{equation}
Choosing a subsequence if needed, we may assume that $(\gl_n)$ converges vaguely to some non-zero $\ga\in\cM_{\le 1}^{(2)}$ (see~\cite[Theorem 4.2]{Kall_RMTA} or~\cite[Section 28]{BillingBookPM}) and
by Lemma~\ref{lem:vagueweak} we may further write $\lambda_n= \alpha_n +\beta_n$ with $\alpha_n$ converging weakly to $\alpha$, and $\beta_n$ converging vaguely to zero, and both measures being concentrated on disjoint sets. By Lemma~\ref{lem:wid-sep-suff}\footnote{The use of Lemma~\ref{lem:wid-sep-suff} corrects a minor gap in \cite{MV2016}. Actually, the analogous step in the proof of \cite[Theorem 3.2, Step 2]{MV2016} does not seem to follow from Lemma 2.4, as it is claimed therein. Indeed, a sequence of sub-probability measures may converge vaguely to zero without disintegrating (think of a Dirac mass escaping to infinity).}, for every $V\in \cF_2^{(2)}$,
\begin{equation}
\lim_{n\to \infty}	\int V(u_1, v_1, u_2, v_2) \alpha_n (\dd u_1, \dd v_1) \beta_n(\dd u_2, \dd v_2) = 0\,.
\end{equation}
If additionally $q=p$, then no mass escapes to infinity and one can choose $\beta_n$ to be zero.
In that case it follows that $\tilde \mu_n \to \alpha$ in $\tilde\cX^{(2)}$, similarly to \cite[Theorem 3.2]{MV2016}.

\par {\bf (ii)} To conclude the result we can now proceed iteratively, in very much the same way as in \cite{MV2016}. Fix a sequence $(\mu_n)$ in $\cM_1^{(2)}$. If $q\in\{0,1\}$ then the considerations in Item (i) above imply the result.
Otherwise, for some sequence $(a_n,a_n)\in \R^d\times\R^d$, at least along a subsequence, we have the decomposition $\mu_n= \alpha_n + \beta_n$, where
\begin{itemize}
\item[\rm{(1)}]$\alpha_n*\delta_{(a_n,a_n)}\to \alpha \neq 0$ weakly as $n\to\infty$;
\item[(2)] For every $V\in \cF_2^{(2)}$ the integral
$$ \int V(u_1,v_1, u_2, v_2) \alpha_n(\dd u_1, \dd v_1) \beta_n(\dd u_2, \dd v_2) $$
converges to zero;
\item[(3)] $\limsup_{n\to\infty}q_{\beta_n}(r) \leq \min\{q,1-q/2\}$;
\item[(4)] $\alpha_n$ and $\beta_n$ have disjoint support.
\end{itemize}
In Item (3) we used for the first inequality that $q_{\beta_n}(r)\le q_{\mu_n}(r)\le q$ and for the second inequality that
\beq
\limsup_{n\to\infty} \gb_n(\bbR^d) \le  \limsup_{n\to\infty} (1-\ga_n(\bbR^d)) {\red \leq} 1 - q/2,
\eeq
by~\eqref{eq:q2} and the fact that $(\gl_n)$ converges {\red vaguely} to $\ga$.
If the limit in Item (3) is not zero one can repeat the procedure in Item (i) to the sequence $(\beta_n)$, i.e., for an appropriate shift of $\beta_n$ one decomposes  it (along some further subsequence) as a sum of two measures where one converges weakly and the other converges vaguely to zero and such that additionally the above four items (1)-(4) are satisfied. If this iteration process terminates after some finite number of stages $k\in\N$, meaning that the limit in Item (3) eventually becomes zero for every $r$, we obtain the decomposition
\begin{equation}\label{eq:decomposition}
	\mu_n= \sum_{j=1}^k \alpha_{n,j}*\delta_{(a_{n,j},a_{n,j})} +\beta_n\,,
\end{equation}
such that $\alpha_{n,j}*\delta_{(a_{n,j},a_{n,j})}$ converges weakly to some non-zero sub-probability measure for each $1\le j\le k$, $q_{\beta_n}(r)\to 0$ for every $r$ and such that for every $V\in\cF_2^{(2)}$ and $1\leq i<j\leq k$,
\beq
\ba
&\int V(u_1+a_{n,i},v_1+a_{n,i},u_2+a_{n,j},v_2+a_{n,j}) \alpha_{n,i}(\dd u_1,\dd v_1)\alpha_{n,j}(\dd u_2,\dd v_2)\\
\text{and} \quad & \int V(u_1+a_{n,i},v_1+a_{n,i},u_2,v_2) \alpha_{n,i}(\dd u_1,\dd v_1) \beta_n(\dd u_2,\dd v_2)
\ea
\eeq
tend to zero as $n\to \infty$. Moreover, by Remark~\ref{rmk:wide-sep} and Lemma~\ref{lem:wid-sep-div-seq}, the $n$-indexed sequences $(a_{n,i})$ satisfy
\beq\label{eq:divergingcenter}
\lim_{n\to\infty} \min_{i\neq j} |a_{n,i}-a_{n,j}| = +\infty.
\eeq
Let us stress that in general the measure $\gb_n$ in~\eqref{eq:decomposition} is not the one appearing in Item (4) (except of course if $k=1$).
If the iteration process does not terminate after a finite number of stages, then one has a similar decomposition that goes by induction. We refer to the proof of~\cite[Theorem 3.2]{MV2016} for details.
\end{proof}

\section{Large Deviation principles}
\label{sec:LDP}
 Let $X=(X_i)_{i\in\bbN_0}$ be a Markov chain in $\bbR^d$ starting from the origin and having transition kernel $\pi$ satisfying the following assumptions:
\begin{enumerate}
\item {\bf (Random walk)} There exists a function $p\colon \bbR^d\mapsto \bbR^+ =[0,\infty)$ and a reference measure $\gl$ such that $\pi(x,A)=\int_A p(y-x)\gl(\dd y)$ for all $x\in\bbR^d$ and Borel sets $A\subseteq \bbR^d$.
\item {\bf (Irreducibility)} For all Borel sets $A\subseteq \bbR^d$ such that $\gl(A)>0$, there exists $k\in\bbN$ such that $\pi^k(x,A) = \int_A p^{* k}(y-x)\gl(\dd y)>0$.
\item {\bf (Tightness)} There exists a positive sequence $(\rho_n)_{n\in \N}$ with $\limsup_{n\to\infty}\tfrac1n\log \rho_n=0$ such that
\beq\label{eq:tails}
\lim_{n\to\infty}\frac{1}{n}\log\bP\Big(\sup_{1\leq i\leq n} |X_i|\geq \rho_n\Big)=-\infty\,.
\eeq
\end{enumerate}
\par Assumption (1) implies shift-invariance of the process, see Remark~\ref{rmk:shift-inv} below. Assumption (2) is used at the end of the proof of Proposition~\ref{pr:PairLDP-LB} when applying the standard Large Deviation lower bound in the usual weak topology~\cite[Corollary 3.4 and Equation (4.1)]{DV76}. Assumption (3) is used during the proof of~Lemma~\ref{lem:PairLDP-UB-1}. These assumptions include many natural examples such as the simple random walk on $\bbZ^d$ (with the counting measure as reference measure) or the discretized Brownian motion $(B_{i\gep})_{i\in\bbN_0}$, where $\gep>0$ and $B$ is Brownian motion (with Lebesgue measure as reference measure). In both cases it is known that it suffices to choose $\rho_n$ to grow superlinearly to infinity in order for \eqref{eq:tails} to be satisfied. More generally, consider an irreducible random walk on $\Z^d$ with i.i.d.\ increments $(Y_\ell)_{\ell\ge 1}$ for which there exist $\delta,\gamma>0$ such that
\begin{equation}
\label{eq:exp-moments-incr}
	\bE[e^{\delta |Y_1|^\gamma}] <\infty\,.
\end{equation}
In that case the reference measure is again the counting measure on $\Z^d$, the chain is irreducible by assumption and it also satisfies the tightness assumption as we will see now. To this end note that, by a union bound and Markov's inequality,
\begin{equation}
	\ba
\bP\Big(\sup_{1\leq i\leq n} |X_i|\geq \rho_n\Big)
&\leq \bP\Big(\bigcup_{i=1}^n \bigcup_{\ell=1}^i |Y_\ell|\geq \rho_n/i\Big)\\
&\leq n^2 \bP\Big(|Y_1|\geq {\rho_n}/{n}\Big)\\
&\leq n^2\exp(-\delta (\tfrac{\rho_n}{n})^\gamma)\bE\Big(\exp(\delta |Y_1|^\gamma)\Big)\,.
\ea
	\end{equation}
To conclude it suffices to choose $\rho_n$ diverging much faster than $n^{(1+\gamma)/\gamma}$. As for the discretized Brownian motion, the same reasoning applies with \(\gamma=1\) and any \(\gd >0\) in~\eqref{eq:exp-moments-incr}.

\par  In the following, $h(\cdot|\cdot)$ is the relative entropy between two sub-probability measures on $\bbR^d\times \bbR^d$, i.e.,
\begin{equation}
	h(\mu|\nu)= 
\begin{cases}
\int \log\Big(\frac{\dd \mu}{\dd \nu}\Big) \dd \mu & \textrm{if } \mu \ll \nu,\\
+\infty& \textrm{else.}
\end{cases}
\end{equation}
Note that:
\begin{enumerate}
\item if $\mu$ is the zero measure it makes sense to let $h(\mu| \nu)= 0$;
\item if $\mu$ and $\nu$ have the same total mass, say $\mathsf{m}\in(0,1)$, then 
\beq
\label{eq:rel-entropy-nonneg}
h(\mu|\nu) = \mathsf{m} h\Big(\frac{\mu}{\mathsf{m}}\Big|\frac{\nu}{\mathsf{m}}\Big)\ge 0,
\eeq
since the relative entropy between two \emph{probability} measures is always non-negative.
\end{enumerate}
Finally, when $\nu \in \cM_1(\bbR^d)$, we write
\beq \label{eq:def-nu-prod-pi}
(\nu \otimes \pi)(\dd x, \dd y) = \nu(\dd x) \pi(x,\dd y).
\eeq
Let us denote by
\beq
L_n^{(2)} := \frac1n \sum_{i=1}^n \gd_{(X_{i-1}, X_i)} \in \cM_1^{(2)}
\eeq
the pair empirical measure associated to $X$ and by $\tilde L_n^{(2)}$ its orbit in $\tilde \cM_1^{(2)}$.
We may now state our main result:
\begin{theorem}\label{thm:LDPpair}
As $n\to\infty$, $(\tilde L_n^{(2)})_{n\in\bbN}$ satisfies a strong Large Deviation principle on the compact metric space $\tilde \cX^{(2)}$ equipped with the metric $\Dtwo$, with speed $n$ and rate function $\tilde J^{(2)}$, where
\beq\label{eq:ratefct}
\tilde J^{(2)}(\xi) := \sum_{\tilde\alpha\in\xi} h(\ga | \ga_{1} \otimes \pi)\,,
\eeq
if $\ga_{1}=\ga_{2}$ for all $\tilde\alpha\in\xi$, and $\tilde J^{(2)}(\xi) := +\infty$ otherwise. Here we write, with a slight abuse of notation, $\ga\in\cM_{\le 1}^{(2)}$ for a representative of $\tilde\ga\in\xi$, while $\ga_{1}$ and $\ga_{2}$ are the projections of the representative $\ga$ onto the first and the last $d$ coordinates respectively.
\end{theorem}
\begin{remark}
\label{rmk:shift-inv}
The rate function in~\eqref{eq:ratefct} is well-defined due to the fact that the transition kernel $\pi$ only depends on the difference between its two arguments. Indeed, this implies that inside the sum in~\eqref{eq:ratefct} the particular choice of $\ga$ among the orbit $\tilde \ga$ does not matter. Moreover, all the terms inside the sum are non-negative by \eqref{eq:rel-entropy-nonneg}. It follows that $\tilde J^{(2)}$ is a non-negative function. Also, note that the relative entropy between two sub-probability measures with the same mass is zero if and only if the measures coincide. Hence, $\tilde J^{(2)}(\xi)=0$ if and only if every orbit $\tilde\alpha\in\xi$ satisfies both $\alpha_1=\alpha_2$ and $\alpha= \alpha_1\otimes \pi$. An elementary computation then shows that $\alpha_1$ (and therefore $\alpha_2$) must be an invariant measure for the Markov chain $X$.
\end{remark}
\begin{remark} Following the same idea, one should be able to prove a \emph{strong} Large Deviation principle for the $k$-th order empirical measure ($k\ge 2$) defined as
\beq
L_n^{(k)} := \frac 1n \sum_{i=1}^n \gd_{(X_i, X_{i+1}, \ldots, X_{i+k-1})},
\eeq
see~\cite[Section 6.5.2]{DemboZei10:book} for the standard version of it.
Naturally, the orbits of the sub-probability measures on $\bbR^{kd}$ would be taken with respect to diagonal shifts of the form $(x, x, \ldots, x)\in \bbR^{kd}$, where $x\in\bbR^d$. In analogy with the case $k=2$, we expect the rate function to read
\beq\label{eq:ratefct-gen}
\xi\in \tilde \cX^{(k)} \mapsto\tilde J^{(k)}(\xi) := \sum_{\ga\in\xi} h(\ga | \ga_{1\ldots k-1} \otimes \pi),
\eeq
if $\ga_{1\ldots k-1}$ (projection onto the $k-1$ first $\bbR^d$-valued coordinates) coincides with $\ga_{2\ldots k}$ (projection onto the $k-1$ last $\bbR^d$-valued coordinates) for all $\alpha\in\xi$, and $\tilde J^{(k)}(\xi) := +\infty$ otherwise. Note that we used a notation analogous to~\eqref{eq:def-nu-prod-pi}, namely:
\beq
\ba
(\ga_{1\ldots k-1} \otimes \pi)(\dd x_1, \ldots, \dd x_k)
&= \ga_{1\ldots k-1}(\dd x_1, \ldots, \dd x_{k-1})\pi(x_{k-1},\dd x_k)\\
&= \ga(\dd x_1, \ldots, \dd x_{k-1}, \bbR^d)\pi(x_{k-1},\dd x_k).
\ea
\eeq
\end{remark}
\par The proof of Theorem~\ref{thm:LDPpair} is split into three parts: we prove the lower semi-continuity of the rate function (Proposition~\ref{pr:lsc}), then the lower bound (Proposition~\ref{pr:PairLDP-LB}) and finally the upper bound (Proposition~\ref{pr:PairLDP-UB}). We shall use the well-known fact that the process $(X_{i-1}, X_i)_{i\in\bbN}$ is itself a Markov chain on $\bbR^d\times\bbR^d$ with transition kernel
\beq\label{eq:kernelY}
\pi^{(2)}(x_1, x_2, \dd y_1, \dd y_2) := \gd_{x_2}(y_1) \pi(x_2, \dd y_2), \qquad x_1, x_2, y_1, y_2\in\bbR^d.
\eeq
\section{Lower semi-continuity of the rate function}
\label{sec:lsc}
\begin{proposition}
\label{pr:lsc}
The function
\beq
\xi\in \tilde \cX^{(2)} \mapsto \tilde J^{(2)}(\xi) 
\eeq
is lower semi-continuous.
\end{proposition}
Before we come to the proof of the above result we need to collect some facts about the relative entropy, which will also be of use later on. To that end we denote by $\cB(\R^{d}\times\R^d)$ the space of bounded and measurable real-valued  functions, by $\cC_b(\R^{d}\times \R^d)$ the space of continuous and bounded real-valued functions and by $\cU(\R^d\times\R^d)=:\cU$ the space of continuous, non-negative and compactly supported functions. 
\begin{proposition} 
\label{pr:alt-form-J-2-alpha}
For every $\ga\in \cM_{\le 1}(\bbR^d\times \bbR^d)$,
\beq
J^{(2)}(\alpha) = \sup_{u\in \cU} \int \log\Big( \frac{u+1}{\pi^{(2)}(u+1)} \Big) \dd \ga\,,
\eeq
where  
\beq
\ba
J^{(2)}(\alpha):=
\begin{cases}
	 h(\ga | \ga_1\otimes \pi)\,,&\text{if $\ga_1=\ga_2$}\,,\\
	+\infty\,,&\text{otherwise.} 
\end{cases}
\ea
\eeq
\end{proposition}
We recall that with a slight abuse of notation, $\ga_{1}$ (resp. $\ga_2$) denotes the projection of $\ga$ onto the first (resp.\ last) $d$ coordinates in the above proposition.
\begin{proof}[Proof of Proposition~\ref{pr:alt-form-J-2-alpha}]
By \cite[Theorem~6.5.12 and Corollary~6.5.10]{DemboZei10:book}\footnote{Although the given reference considers only \emph{probability} measures, \eqref{eq:rel-ent-as-sup} carries over to sub-probability measures. Indeed, if $\ga \in \cM_{\leq 1}^{(2)}$ and $\mathsf{m}= \ga(\bbR^d\times \bbR^d)\in (0,1)$ then $J^{(2)}(\ga) = \mathsf{m} J^{(2)}(\ga/\mathsf{m})$.} one has for every sub-probability measure $\ga$ the relation
\beq\label{eq:rel-ent-as-sup}
J^{(2)}(\alpha) = \sup_{v\in \cB(\R^{d}\times\R^d), v\geq 1} \int \log\Big( \frac{v}{\pi^{(2)}v} \Big) \dd \ga \,.
\eeq
Moreover, by~\cite[Exercise 6.5.7]{DemboZei10:book} we have that 
\beq
\sup_{v\in \cB(\R^{d}\times\R^d), v\geq 1} \int \log\Big( \frac{v}{\pi^{(2)}v} \Big) \dd \ga  = \sup_{v\in \cC_b(\R^{d}\times\R^d), v\geq 1} \int \log\Big( \frac{v}{\pi^{(2)}v} \Big) \dd \ga\,.
\eeq
Writing $v\in \cC_b(\R^{d}\times \R^d)$ with $v\geq 1$ as $v=u+1$ for $u$ a non-negative, bounded and continuous function, an approximation argument using smooth cut-off functions and dominated convergence shows that 
\beq
\sup_{v\in \cC_b(\R^{d}\times\R^d), v\geq 1} \int \log\Big( \frac{v}{\pi^{(2)}v} \Big) \dd \ga = \sup_{u\in \cU} \int \log\Big( \frac{u+1}{\pi^{(2)}(u+1)} \Big) \dd \ga\,.
\eeq
\end{proof}

We are now in a position to prove Proposition~\ref{pr:lsc}.

\begin{proof}[Proof of Proposition~\ref{pr:lsc}] 
Let $(\mu_n)$ be a sequence in $\tilde\cX^{(2)}$ converging to $\xi =\{\tilde \ga_i\}_{i\in I}$. Suppose that there exists $\ell\in(0,\infty)$ such that $\tilde J^{(2)}(\mu_n)\le \ell$ for all $n$ large enough and let us show that $\tilde J^{(2)}(\xi)\le \ell$. We for now restrict to the case where for each $n$, $\mu_n$ is made of a single orbit, so that $\tilde J^{(2)}(\mu_n)=J^{(2)}(\mu_n)$. Let $\gep>0$. 
By the same arguments used in Proposition~\ref{pr:D2compact} and possibly restricting to a subsequence, we may write for some $k\ge 1$, 
\beq
\mu_n = \sum_{i=1}^k \ga_n^{(i)} + \gb_n,
\eeq
where $\ga_n^{(i)}$ ($1\le i \le k$) and $\gb_n$ are sequences of sub-probability measures in $\bbR^d \times \bbR^d$ such that
\beq\label{eq:conv-ga-i-n}
\ga_n^{(i)}*\delta_{(a_n^{(i)},a_n^{(i)})} \stackrel{\rm (weakly)}{\longrightarrow}\ga_i \in \tilde \ga_i, \qquad n\to \infty,
\eeq
for some sequences $a_n^{(i)}$ ($1\le i\le k$) in $\bbR^d$, and $(\gb_n)$ is widely separated from each $(\ga_n^{(i)})$, with (recall~\eqref{eq:conc-fct})
	\beq
	\limsup_{n\to\infty} q_{\gb_n}(r) \le \gep, \qquad \forall r>0,
	\eeq
and
\beq
\lim_{n\to\infty} \min_{i\neq j} |a_n^{(i)}-a_n^{(j)}| = +\infty.
\eeq
Then, by Lemma~\ref{lem:vagueweak} and the construction in Step 2 of Proposition~\ref{pr:D2compact}, there exists for each $i$ a sequence $(R_n^{(i)})_{n\in\N_0}$ tending to infinity such that 
\beq \label{eq:ppties-supports}
\ba
&\supp(\alpha_n^{(i)})\subseteq B((-a_n^{(i)}, -a_n^{(i)}), R_n^{(i)}),\\
&i\neq j \implies B((-a_n^{(i)}, -a_n^{(i)}), R_n^{(i)}) \cap B((-a_n^{(j)}, -a_n^{(j)}), R_n^{(j)}) = \emptyset,
\ea
\eeq
and the support of $\beta_n$ is contained in the complement of $\bigcup_i B((-a_n^{(i)}, -a_n^{(i)}), R_n^{(i)})$. 
In particular, $\alpha_n^{(1)}, \alpha_n^{(2)},\ldots, \beta_n$ are concentrated on disjoint sets.
We now define
	\beq\label{eq:def-u-i-n}
	u_i^{(n)}(x,y) = u_i(x+a_n^{(i)},y+a_n^{(i)}), \qquad x,y\in\bbR^d\,,
	\eeq
where, for each $i$, $u_i$ is a given non-negative continuous function with compact support.
Consequently, the support of each $u_i^{(n)}$ is contained in a compact ball in $\R^{2d}$ centered around $(-a_n^{(i)}, -a_n^{(i)})$. In particular, recalling~\eqref{eq:kernelY} there is $R>0$ such that for all $1\leq i\leq k$,
	\beq
	\ba
	&\supp\ u_i^{(n)} \subseteq\ B((-a_n^{(i)},-a_n^{(i)}), R)\,,\\
	&\supp\ \pi^{(2)}u_i^{(n)} \subseteq\ \bbR^d\times B(-a_n^{(i)}, R)\,,
	\ea
	\eeq
	hence
	\beq\label{eq:supportui}
	\Big( \frac{1+u_i^{(n)}}{1+\pi^{(2)}u_i^{(n)}} \Big)(x,y) \neq 1 \Rightarrow y\in B(-a_n^{(i)}, R).
	\eeq
	We may now check that
	\beq
		\label{eq:lsc-weaklimit}
		\lim_{n\to\infty } \int \log\Big( \frac{1+u_i^{(n)}}{1+\pi^{(2)}u_i^{(n)}} \Big) \dd \ga_n^{(j)}
		= 
		\begin{cases}
			\int \log\Big( \frac{1+u_i}{1+\pi^{(2)}u_i} \Big) \dd \ga_i & (i=j)\\
			0 & (i\neq j).
		\end{cases}
		\eeq
		The case $i=j$ follows from~\eqref{eq:conv-ga-i-n} and~\eqref{eq:def-u-i-n} together with the fact that $\pi$ only depends on the difference of its arguments. Assume now that $i\neq j$. Note that, from what precedes, the integral on the left-hand side may be restricted to couples $(x,y)$ such that $y\in B(-a_n^{(i)}, R) \cap B(-a_n^{(j)}, R_n^{(j)})$. If this set is non-empty, then $|a_n^{(i)}-a_n^{(j)}| \le R + R_n^{(j)}$. This is not possible when $n$ is large enough, since $|a_n^{(i)}-a_n^{(j)}|> R_n^{(i)} + R_n^{(j)}$ by~\eqref{eq:ppties-supports}, so the integral is eventually zero.
		Finally, letting
		\beq
		u^{(n)} = \sum_{1\le i\le k} u_i^{(n)},
		\eeq
		we obtain:
		\beq\label{eq:lowerlast}
		\ba
		\ell &\ge \liminf_{n\to\infty} J^{(2)}(\mu_n)\ge \liminf_{n\to\infty} \int \log\Big( \frac{1+u^{(n)}}{1+\pi^{(2)}u^{(n)}} \Big) \dd \mu_n\\
		&\ge \liminf_{n\to\infty} \sum_{1\le i,j\le k} \int \log\Big( \frac{1+u_i^{(n)}}{1+\pi^{(2)}u_i^{(n)}} \Big) \dd \ga_n^{(j)} \\
		&\ge \sum_{1\le i\le k} J^{(2)}(\ga_i) - \gep.
		\ea
		\eeq
	Here, the contribution coming from $\beta_n$ is zero because for sufficiently large $n$ we have that $\beta_n(\dd x,\dd y)=0$ for $y\in B(-a_n^{(i)}, R)$, which together with~\eqref{eq:supportui} shows that the contribution is indeed zero. It now remains to send $\gep$ to zero and $k$ to infinity. Finally, to treat the general case, that is when $\mu_n$ has possibly more than one orbit, the idea is the same as in the last paragraph of~\cite[proof of Lemma 4.2]{MV2016}.
\end{proof}

\section{Lower bound}
\label{sec:lb}

\begin{proposition}\label{pr:PairLDP-LB}
For any open set $G$ in $\tilde\cX^{(2)}$, 
\beq
\liminf_{n\to\infty} \frac1n \log \bP(\tilde L_n^{(2)} \in G) \ge - \inf_{\xi\in G} \tilde J^{(2)}(\xi).
\eeq
\end{proposition}

\begin{proof}[Proof of Proposition~\ref{pr:PairLDP-LB}]
Let $\xi = \{\ga_i, i\in I\}$ be an element of $\tilde\cX^{(2)}$ such that $\tilde J^{(2)}(\xi)<+\infty$, hence $h(\ga_i | \ga_{i,1} \otimes \pi)<+\infty$ for all $i\in I$. Let $U$ be any open neighborhood of $\xi$. It is enough to prove that
\beq
\liminf_{n\to\infty} \frac1n \log \bP(\tilde L_n^{(2)} \in U) \ge - \tilde J^{(2)}(\xi).
\eeq
We proceed as in~\cite[Lemma 4.3]{MV2016} and use the density of $\tilde \cM_1^{(2)}$ in $\tilde \cX^{(2)}$, see Proposition~\ref{pr:D2compact}. Let $k\ge 1$ and consider the sequence $(\mu_M)$ defined as in~\eqref{eq:explicitmu}, except that we replace the totally disintegrating sequence $(\gl_M)$ by $(\gl_{M,1}\otimes \pi)$, that is totally disintegrating as well, since for every $x\in\bbR^d$ and $r>0$, $(\gl_{M,1}\otimes \pi)(B((x,x),r)) \le \gl_{M,1}(B(x,r))$ (recall that we use the product norm on $\bbR^d\times \bbR^d$). Moreover, it
is such that $\tilde\mu_M$ still converges in $\tilde\cX^{(2)}$ to $\xi$. Note that $\nu\in\cM_{\le1}^{(2)} \mapsto h(\nu| \nu_1\otimes \pi)$ is sub-additive as the supremum of linear functions, see~\eqref{eq:rel-ent-as-sup}. Therefore, we obtain
\beq
J^{(2)} (\tilde\mu_M)=h(\mu_M | \mu_{M,1}\otimes \pi) \le \sum_{1\le i \le k}h(\ga_i | \ga_{i,1}\otimes \pi)\leq J^{(2)}(\xi).
\eeq
Here, we used that $h(\lambda_{M,1}\otimes \pi|\lambda_{M,1}\otimes \pi)=0$ to obtain the first inequality, while the last inequality follows from the next-to-last sentence in Remark~\ref{rmk:shift-inv}.
Thus, we have shown that there exists a sequence $(\mu_M)_{M\in\N_0}$ in $\tilde \cM_1^{(2)}$ which converges in $\tilde \cX^{(2)}$ to $\xi$ and is such that
	\beq
	\limsup_{M\to\infty}J^{(2)}(\tilde\mu_M)\leq J^{(2)}(\xi)\,.	
	\eeq
The lower bound now follows from the standard Large Deviation lower bound of the pair empirical measure on $\cM_1^{(2)}$ (see~\cite{DV76} and the discussion in Section~\ref{sec:intro}).	\end{proof}

\section{Upper bound}
\label{sec:ub}

In this section we prove the following
\begin{proposition}\label{pr:PairLDP-UB}
For any closed set $F$ in $\tilde\cX^{(2)}$, 
\beq
\limsup_{n\to\infty} \frac1n \log \bP(\tilde L_n^{(2)} \in F) \le - \inf_{\xi\in F} \tilde J^{(2)}(\xi).
\eeq
\end{proposition}
Recall that $\cU$ denotes the space of non-negative, continuous and compactly supported functions defined on $\bbR^d\times \bbR^d$. For any $k\ge 1$, $u := (u_1,\ldots, u_k)\in \cU^k$ and $a:=(a_1, \ldots, a_k) \in (\bbR^d)^k$, let $g = g(u,a,R) \colon \bbR^d\times \bbR^d \to (0,\infty)$ be defined by
\beq
\label{eq:defg}
g(x,y) = 1 + \sum_{i=1}^k u_i(x+a_i, y+a_i) \varphi\Big(\frac{x+a_i}{R}, \frac{y+a_i}{R}\Big), \qquad x,y\in\bbR^d,
\eeq
where $\varphi$ is a smooth non-negative function such that $0\le \varphi \le 1$, $\varphi = 1$ inside the unit ball and $\varphi=0$ outside the ball of radius two. Recall the definition of $\pi^{(2)}$ in \eqref{eq:kernelY} and define, for every $\mu\in\cM_1(\bbR^d \times \bbR^d)$,
\beq
{\sf F}(u,R, \mu)= \suptwo{a_1,\ldots, a_k}{\min_{i\neq j} |a_i - a_j| \ge 4R} \int_{\bbR^d\times \bbR^d} -\log\Big(\frac{\pi^{(2)}g(x,y)}{g(x,y)} \Big) \mu(\dd x, \dd y).
\eeq
Since ${\sf F}(u,R, \cdot)$ is invariant under shifts of the form $\mu \to \mu * \gd_{(x,x)}$, we may lift it up to a function $\tilde {\sf F}$ defined on $\tilde\cM_1^{(2)}$. In the sequel, we write
\beq
u_{i,R}(x,y) := u_i(x,y) \varphi(x/R, y/R), \qquad x,y\in\bbR^d.
\eeq

The proof of the upper bound follows from the following three lemmas:

\begin{lemma}[Sub-exponential growth]
\label{lem:PairLDP-UB-1} 
For any choice of $k\ge 1$, $u\in \cU^k$ and $R>0$,
\beq
\limsup_{n\to\infty} \frac1n \log \bE \exp(n \tilde {\sf F}(u,R, \tilde L_n^{(2)})) \le 0.
\eeq
\end{lemma}

\begin{lemma}[Lower-semicontinuous extension]
\label{lem:PairLDP-UB-2}
If the sequence $(\tilde \mu_n)$ converges to $\xi = (\tilde \ga_i)_{i\in I}$ in $(\tilde\cX^{(2)}, {\mathbf D}_2)$, then for every finite $k\le |I|$, $u=(u_1,\ldots, u_k)\in \cU^k$ and $R>0$,
\beq
\liminf_{n\to \infty} \tilde {\sf F}(u,R, \tilde \mu_n) \ge \tilde\Lambda(u,R,\xi),
\eeq
where
\beq
\label{eq:tildeLambda}
\tilde\Lambda(u,R,\xi):=
 \sup_{\{\tilde\ga_1, \ldots, \tilde\ga_k\}\subseteq \xi} 
\sum_{i=1}^k  \sup_{b\in\bbR^d} \int 
-\log\Big\{\frac{\pi^{(2)}(1+u_{i,R})(x,y)}{(1+u_{i,R})(x,y)}\Big\}
 \ga_i(\dd x  + b, \dd y + b).
\eeq
\end{lemma}
 \begin{remark}
\label{rmk:PairLDP-UB-2}
Lemma~\ref{lem:PairLDP-UB-2} is analogous to Lemma 4.6 in~\cite{MV2016}. However, the two suprema in~\eqref{eq:tildeLambda} are not in the original paper. First, we add the supremum over $b\in\bbR^d$ so that the quantity inside is a function of the orbit $\tilde \ga_i$ only rather than a function of a particular member of its orbit. This has however no consequence on the sequel of the argument in~\cite{MV2016}, since they later consider a supremum over functions (Lemma 4.7) allowing for arbitrary shifts. The other supremum is here to stress that an element of $\xi$ is a \emph{collection} of sub-probability orbits rather than a sequence.
\end{remark}

\begin{lemma}\label{lem:PairLDP-UB-3}
We have
\beq
\tilde J^{(2)}(\xi) = \suptwo{R>0,\, 1\le k\le |I|}{u_1, \ldots, u_k\in \cU} \tilde \Lambda(u, R, \xi).
\eeq
\end{lemma}

\begin{proof}[Proof of Lemma~\ref{lem:PairLDP-UB-1}]
Use that
\beq
- n \int_{\bbR^d\times \bbR^d} \log\Big(\frac{\pi^{(2)}g(x,y)}{g(x,y)} \Big) L_n^{(2)}(\dd x, \dd y)=
\log \prod_{i=1}^n \frac{g(X_{i-1}, X_i)}{\pi^{(2)}g(X_{i-1}, X_i)}\,,
\eeq
so that
\beq\label{eq:martingale}
\bE\Big[\exp\Big(n\int_{\R^d\times\R^d}-\log\Big(\frac{\pi^{(2)}g(x,y)}{g(x,y)}\Big)L_n^{(2)}(\dd x,\dd y)\Big)\Big]
= \bE\Big[\prod_{i=1}^n\frac{g(X_{i-1}, X_i)}{\pi^{(2)}g(X_{i-1},X_i)}\Big]\,.
\eeq
We write the product as 
\beq
\prod_{i=1}^n\frac{g(X_{i-1}, X_i)}{\pi^{(2)}g(X_{i-1},X_i)}
= \frac{g(X_0,X_1)}{\pi^{(2)}g(X_{n-1},X_n)}\prod_{i=1}^{n-1}\frac{g(X_i,X_{i+1})}{\pi^{(2)}g(X_{i-1}, X_i)}\,,
\eeq
and since $g$ is bounded from below by $1$ we see that
\beq
\bE\Big[\prod_{i=1}^n\frac{g(X_{i-1}, X_i)}{\pi^{(2)}g(X_{i-1},X_i)}\Big]
\leq (\sup g)\ \bE\Big[\prod_{i=1}^{n-1}\frac{g(X_i,X_{i+1})}{\pi^{(2)}g(X_{i-1}, X_i)}\Big]\,,
\eeq
and, by the fact that $Y_i= (X_{i-1}, X_i)$ is a Markov chain and an induction argument the last expectation is one.
 This shows that the exponential rate of the right-hand side of~\eqref{eq:martingale} is zero. It therefore only remains to deal with the case in which in~\eqref{eq:martingale} an additional supremum is taken over $a_1,\ldots, a_k$ as in the statement of the result. This follows via a coarse graining argument. The idea is that by~\eqref{eq:tails} it is exponentially unlikely that $X$ travels in the time interval $[0,n]$ to a distance $\rho_n$, which allows one to restrict the supremum over $a_1,\ldots,a_k$ to balls of radius $\rho_n$. In a very similar way as in~\cite[proof of Lemma 4.5]{MV2016} one may then conclude, so we omit the details.

\end{proof}

\begin{proof}[Proof of Lemma~\ref{lem:PairLDP-UB-2}] Let $k\le |I|$ be a finite integer and $\{\tilde \ga_1, \ldots, \tilde \ga_k\} \subseteq \xi$. As it can be seen from the second step of the proof of Proposition~\ref{pr:D2compact} (see in particular~\eqref{eq:decomposition}--\eqref{eq:divergingcenter}), convergence in ${\mathbf D_2}$ implies the existence of a decomposition 
\beq\label{eq:decompupper}
\mu_n = \sum_{j=1}^k \ga_{n,j} * \gd_{(a_{n,j},a_{n,j})} + \gb_n,
\eeq
along subsequences
as in~\eqref{eq:decomposition}, where, for all $1\le j\le k$ 
\begin{itemize}
\item $(a_{n,j})_{n\ge 1}$ is a sequence in $\bbR^d$ satisfying
\beq\label{eq:FarApart}
|a_{n,i} - a_{n,j}| \ge 4R
\eeq
if $n$ large enough and $i\neq j$;
\item $\ga_{n,j} * \gd_{(a_{n,j},a_{n,j})}$ converges weakly to $\ga_j$ as $n\to\infty$, where $\ga_j$ is some element in the orbit of $\tilde \ga_j$;
\item $(\ga_{n,j})$ and $(\gb_n)$ are widely separated.
\end{itemize}
Recall~\eqref{eq:defg}. Choosing $a_i = - a_{n,i}$ in the definition of $g$, we obtain
\beq
g(x,y) = 1 + \sum_{i=1}^k u_{i,R}(x-a_{n,i}, y-a_{n,i}).
\eeq
By~\eqref{eq:FarApart} and our assumption on $\varphi$, at most one term in the sum above can be nonzero. Also,
\beq
\pi^{(2)}u_{i,R}(x-a_{n,i}, y-a_{n,i})=
\int u_{i,R}(y-a_{n,i}, z-a_{n,i})p(z-y) \lambda(\dd z)
\eeq
is nonzero for at most one value of $1\le i \le k$ which is the same as in the above. We finally obtain:
\beq
\log\Big(\frac{\pi^{(2)}g(x,y)}{g(x,y)}\Big) = \sum_{i=1}^k 
\log\Big(\frac
{1 + \pi^{(2)}u_{i,R}(x-a_{n,i}, y-a_{n,i})}
{1 + u_{i,R}(x-a_{n,i}, y-a_{n,i})}
\Big)
\eeq%
and we can conclude almost as in~\cite[Lemma 4.6]{MV2016}. As we already pointed out in Remark~\ref{rmk:PairLDP-UB-2}, there is more flexibility in choosing $a_i$ (see~\eqref{eq:defg}), which explains the additional supremum (over $b$) in our statement, compared to~\cite[Lemma 4.6]{MV2016}. Indeed, let $b_j\in\bbR^d$ for all $1\le j \le k$. Then we may choose $a_i = - a_{n,i}+b_i$ instead of $a_i=-a_{n,i}$. If we require that $|a_{n,i} - a_{n,j}| \ge 4R + \max_{1\le j \le k} |b_j|$ instead of simply $|a_{n,i} - a_{n,j}| \ge 4R$, then we finally get our claim. To see why the second supremum appears, note that the choice of $\{\tilde\alpha_1, \tilde\alpha_2, \ldots, \tilde\alpha_k\}$ at the beginning of the proof was arbitrary.
\end{proof}

\begin{proof}[Proof of Lemma~\ref{lem:PairLDP-UB-3}]
Let $\xi = \{\tilde \ga_i\}_{i\in I} \in \tilde\cX^{(2)}$, $1\le k\le |I|$ and $u=(u_1, \ldots, u_k)\in \cU^k$.
Note that, for every $1\le i \le k$,
\beq
R\in(0,\infty)\mapsto \sup_{u_i\in \cU} \int - \log\Big\{ \frac{\pi^{(2)}(1+u_{i,R})(x,y)}{(1+u_{i,R})(x,y)}\Big\} \ga_i(\dd x, \dd y)
\eeq		
converges non-decreasingly, as $R\to \infty$, to
\beq
\sup_{u\in \cU} \int - \log\Big\{ \frac{\pi^{(2)}(1+u)(x,y)}{(1+u)(x,y)}\Big\} \ga_i(\dd x, \dd y)\, .
\eeq
Hence,	 
\beq
\ba
\suptwo{R>0}{u_1, \ldots, u_k\in \cU} \tilde \Lambda(u, R, \xi) &= \sup_{\{\tilde\ga_1, \ldots, \tilde\ga_k\}\subseteq \xi} 
\sum_{i=1}^k \suptwo{b\in\R^d}{u\in \cU} \int - \log\Big\{ \frac{\pi^{(2)}(1+u)(x,y)}{(1+u)(x,y)}\Big\} \ga_i(\dd x+b, \dd y+b)\\
&=\sup_{\{\tilde\ga_1, \ldots, \tilde\ga_k\}\subseteq \xi} \sum_{i=1}^k \sup_{u\in \cU} \int - \log\Big\{ \frac{\pi^{(2)}(1+u)(x,y)}{(1+u)(x,y)}\Big\} \ga_i(\dd x, \dd y)\,.
\ea
\eeq
Here, the second equality holds since instead of shifting the measure $\ga_i$ one might also consider a spatial shift of $u\in\cU$.
The supremum in the sum  coincides with $J^{(2)}(\ga_i)$, by Proposition~\ref{pr:alt-form-J-2-alpha}. Hence, to conclude it suffices to consider the supremum over $1\leq k\leq |I|$.
\end{proof}

\begin{proof}[Proof of Proposition~\ref{pr:PairLDP-UB}]
The proof follows from Lemmas~\ref{lem:PairLDP-UB-1}-\ref{lem:PairLDP-UB-3} in the exact same way as Proposition 4.4 in~\cite{MV2016} follows from Lemmas 4.5--4.7 therein. 

We sketch the argument for the reader's convenience. Let $\xi \in\tilde \cX^{(2)} = \{\tilde \ga_i\}_{i\in I}$, $\gd>0$ and $B(\xi,\gd)$ be the associated closed ball for the ${\mathbf D}_2$ metric. By Lemma~\ref{lem:PairLDP-UB-1}, we have for every $1\le k\le |I|$, $u\in \cU^k$ and $R>0$,
\beq
\ba
0 &\ge \limsup_{n\to\infty} \frac1n \log \bE\Big[ \exp\Big(n \tilde {\sf F}(u,R, \tilde L_n^{(2)})\Big)\ind_{\{\tilde L_n^{(2)}\in B(\xi,\gd)\}}\Big]\\
& \ge \inf_{B(\xi,\gd)} \tilde {\sf F}(u,R, \cdot) + \limsup_{n\to\infty} \frac1n \log \bP\big(\tilde L_n^{(2)}\in B(\xi,\gd)\big).
\ea
\eeq
Letting $\gd\to 0$, Lemma~\ref{lem:PairLDP-UB-2} yields:
\beq
\limsup_{\gd\to0}\limsup_{n\to\infty} \frac1n \log \bP\big(\tilde L_n^{(2)}\in B(\xi,\gd)\big) \le - \tilde \Lambda(u,R,\xi).
\eeq
Optimizing over $(k,u,R)$, we finally get, by Lemma~\ref{lem:PairLDP-UB-3}:
\beq
\limsup_{\gd\to0}\limsup_{n\to\infty} \frac1n \log \bP\big(\tilde L_n^{(2)}\in B(\xi,\gd)\big) \le - \tilde J^{(2)}(\xi).
\eeq
Since $\tilde \cX^{(2)}$ is compact (Proposition~\ref{pr:D2compact}), this is enough to conclude the proof.
\end{proof}

\section{Adaptation to rescaled random walks}
\label{sec:adapt}
A small adaptation of the proof of Proposition~\ref{pr:PairLDP-UB} yields the same result also for a rescaled random walk. To precisely formulate the result we need to introduce more notation. Let $(X_i)_{i\in\bbN_0}$ be a random walk in $\bbZ^d$. Assume that its step distribution is centered and square-integrable, with $(1/d){\rm Id}$ as covariance matrix and that a tightness condition as in~\eqref{eq:tails} holds. Let $(a_n)$ be a sequence of positive real numbers converging to $+\infty$ and such that $a_n^2=o(n)$. Let $\gep>0$. Define $\ell:= \ell(\gep,n) = \lfloor \gep a_n^2\rfloor$ and 
\beq
{X}_i^{\gep,n}:=\frac{X_{i\ell}}{a_n}, \qquad i\in\bbN_0.
\eeq
Denote by $L_{n,\gep}^{(2)}$ the corresponding pair empirical measure,
that is
\beq
L_{n,\gep}^{(2)} = \frac1{\lfloor n/\ell \rfloor} \sum_{i=1}^{\lfloor n/\ell \rfloor} \gd_{(X_{i-1}^{\gep,n},X_i^{\gep,n})}.
\eeq
\begin{remark}
The relevant scale for potential applications to the Swiss cheese model~\cite{Phetpradap, BBH2001} corresponds to the choice $a_n := n^{1/d}$.
\end{remark}
Before moving on, let us observe that for every $\gep>0$ and $n\in \bbN$, the rescaled Markov chain $(X_i^{\gep,n})_{i\in \bbN_0}$ is itself an irreducible random walk that is supported on the appropriate sub-lattice of $({\bbZ^d}/{a_n})$, thus it satisfies Assumptions (1) and (2) from Section~\ref{sec:LDP}. The relevance of Assumption (3) therein shall be discussed when completing the proof of the main result of this section, that is:
\begin{theorem}
\label{cor:PairLDP-UB-scaled}
For any closed set $F$ in $\tilde\cX^{(2)}$ and $\gep>0$, 
\beq
\limsup_{n\to\infty} \frac{a_n^2}{n} \log \bP(\tilde L_{n,\gep}^{(2)} \in F) \leq- \frac{1}{\gep}  \inf_{\xi\in F} \tilde J_{\gep/d}^{(2)}(\xi).
\eeq
 where $\tilde J_\gep^{(2)}$ is defined as in \eqref{eq:ratefct} with $\pi$ replaced by $\pi_\gep$, that is the Brownian semi-group at time $\gep$.
\end{theorem}
\begin{proof}[Proof of Theorem~\ref{cor:PairLDP-UB-scaled}]
It turns out that only the proof and the statement of Lemma~\ref{lem:PairLDP-UB-1} need to be adapted. The remaining statements are about the rate function rather than the Markov chain at hand. To that end we define
	\beq
	\pi_{n,\gep}^{(2)}g(x,y)= \bE\Big[g(y,y+X_1^{\gep,n})\Big]\,,\qquad\text{and}\qquad
	\pi_{\gep}^{(2)}g(x,y)= \bE\Big[g(y,y+B_\gep)\Big]\,.
	\eeq
With these notations at hand we define ${\sf F}$ as in Section~\ref{sec:ub} but with $\pi^{(2)}$ replaced by $\pi_\gep^{(2)}$. Then, defining $M:=\frac{n}{\ell}=\tfrac{n}
{\lfloor \gep a_n^2\rfloor}$ we show that
\beq
\limsup_{n\to\infty} \frac1M \log \bE \exp(M \tilde {\sf F}(u,c,R, \tilde L_{n,\gep}^{(2)})) \le 0\,.
\eeq
Note that the result would be immediate from the proof of Lemma~\ref{lem:PairLDP-UB-1} if $\pi_\gep^{(2)}$ would be replaced by the transition kernel $\pi_{n,\gep}^{(2)}$ of $X^{\gep,n}$.
Following the proof of Lemma~\ref{lem:PairLDP-UB-1}, we write
\beq\label{eq:prod-decomp}
\prod_{i=1}^{M}\frac{g(X_{i-1}^{\gep,n}, X_i^{\gep,n})}{\pi_\gep^{(2)}g(X_{i-1}^{\gep,n}, X_i^{\gep,n})}
= \prod_{i=1}^{M}\frac{g(X_{i-1}^{\gep,n}, X_i^{\gep,n})}{\pi_{n,\gep}^{(2)}g(X_{i-1}^{\gep,n}, X_i^{\gep,n})}\prod_{i=1}^{M}\frac{\pi_{n,\gep}^{(2)}g(X_{i-1}^{\gep,n}, X_i^{\gep,n})}{\pi_{\gep}^{(2)}g(X_{i-1}^{\gep,n}, X_i^{\gep,n})}\,.
\eeq
Since (i) $g$ is positive, continuous and constant outside of a compact set and (ii) $X_1^{\gep,n}$ converges in law to $B_{\gep/d}$ as $n\to\infty$, it follows that
\beq
\label{eq:unif-conv-g}
\lim_{n\to\infty}\sup_{x,y\in \R^d} \left|\frac{\pi_{n,\gep}^{(2)}g(x,y)}{\pi_{\gep/d}^{(2)}g(x,y)}-1\right| = 0\,.
\eeq 
The variable $x$ in the supremum plays no actual role and uniformity in $y$ can be deduced from the uniform continuity of $g$ and a coupling for which $X_1^{\gep,n}$ converges  to $B_\gep$ almost surely as $n\to\infty$. The existence of such a coupling is guaranteed by Skorohod's representation theorem~\cite[Chapter 1, Theorem 6.7]{BillingBookCvg}. The convergence in~\eqref{eq:unif-conv-g} allows to control the rightmost factor in~\eqref{eq:prod-decomp}. The first factor on the right-hand side can be dealt with as in the proof of Lemma~\ref{lem:PairLDP-UB-1}, provided that a tightness property similar to~\eqref{eq:tails} holds. This is indeed the case, since
\beq
\bP\big(\sup_{1\le i \le M} |X_i^{\gep,n}|\ge \rho'_M\big)
=\bP\big(\sup_{1\le i \le M} |X_{i\ell}|\ge a_n\rho'_M\big)
\le \bP\big(\sup_{1\le i \le n} |X_i|\ge a_n\rho'_M\big),
\eeq
which decreases super-exponentially fast, as $n\to\infty$, if we pick $\rho'_M := \rho_n / a_n$, with $\rho_n$ a diverging sequence which ensures tightness of the underlying random walk $(X_i)_{i\in\bbN_0}$.
\end{proof}

\section{From the pair empirical measure to the empirical measure}\label{sec:contraction}\label{sec:contraction}
Consider a Markov chain $(X_i)_{i\in \N_0}$ satisfying the same assumptions as in Section~\ref{sec:LDP} and define the empirical measure $L_n$ of $(X_i)_{i\in\N_0}$ as in~\eqref{def:empr-meas}.
Note that $L_n$ is simply the second marginal of the pair empirical measure $L_n^{(2)}$ defined in~\eqref{def:pair-empr-meas}.
We further denote by $\tilde L_n$ the orbit of $L_n$ in $\tilde \cM_1(\R^d)$, the space defined in~\cite[Section 2]{MV2016}, see Remark~\ref{rmk:not-the-same} for the difference between $\tilde \cM_1^{(2)}$ and $\tilde \cM_1(\R^{2d})$. We moreover denote by $\tilde\cX$ the compactification of $\tilde\cM_1(\R^d)$, equipped with the metric $\bfD$, see~\cite[Section 3]{MV2016} for details or Appendix~\ref{sec:a} for a short overview. Given a sub-probability measure $\mu\in\cM_{\leq 1}(\R^d\times\R^d)$ we denote by $\Pr(\mu)$ its projection onto the second marginal (of course, the sequel applies to the projection onto the first marginal as well). We then define a projection map $\tilde{\Pr}:\tilde \cX^{(2)}\to \tilde\cX$ by first defining 
\begin{equation}
	\tilde \Pr (\tilde \mu)= \{\Pr(\mu)*\delta_x\,:\, x\in \R^d\}
\end{equation}
for $\tilde\mu\in \tilde\cM_{\leq 1}(\R^d\times\R^d)$, then
\begin{equation}
\tilde \Pr(\xi)=\{\tilde{\Pr}(\tilde\alpha_i)\}_{i\in I}
\end{equation}
for $\xi= \{\tilde\alpha_i\}_{i\in I}\in\tilde \cX^{(2)}$. The goal of this section is to derive a Large Deviation principle for $\tilde L_n$ from one of $\tilde L_n^{(2)}$, which is done in Theorem~\ref{thm:LDPempirical} below. By the contraction principle it would be sufficient to check continuity of the projection map. This unfortunately turns out to be false as the following counter-example shows. 
Consider the sequence of measures given by $\mu_n=\delta_{(n,0)}$. Then, on the one hand we have that for any $f\in \cF_k^{(2)}$
\begin{equation}
	\int f(u_1, v_1, \ldots, u_k, v_k)\prod_{1\leq i\leq k}\delta_{(n,0)}(\dd u_i, \dd v_i)
	= f(n,0,\ldots, n,0)\,.
\end{equation}
Since $f\in\cF_k^{(2)}$ the latter quantity converges to zero as $n\to\infty$. This shows that the sequence $\tilde\mu_n$ converges to $\emptyset$ in $(\tilde\cX^{(2)},{\bf D}_2)$.
On the other hand, 
\begin{equation}
	\tilde \Pr(\tilde\mu_n)= \{\delta_x\,:\, x\in\R^d\}\,,
\end{equation}
which is a constant sequence and therefore converges to $ \{\delta_x\,:\, x\in\R^d\}$. In particular,
\begin{equation}\label{eq:c-ex}
	\lim_{n\to\infty} \tilde \Pr(\tilde\mu_n)\neq \tilde \Pr(\lim_{n\to\infty}\tilde\mu_n)\,,
\end{equation}
which shows that $\tilde \Pr$ is not continuous and we cannot apply the standard contraction principle. 
There are of course generalizations of the classical contraction principle, however it is not clear to us how to use them in the present context. The key to still obtain a Large Deviation principle is via a smoothing of the projection map. More precisely, for $A>0$, we define  $\phi_A:\R^d\to [0,1]$ to be a smooth function such that
\begin{equation}
\begin{aligned} \label{eq:def-phi-n}
\phi_A(x)=\begin{cases}
1\,, & |x|\leq  A\,,\\
0\,, &|x|\geq  2A\,.
\end{cases}
\end{aligned}
\end{equation}
We then define for any sub-probability measure $\mu\in\cM_{\leq 1}(\R^d\times\R^d)$
\begin{equation}
({\Pr}_A\mu)(B)=\int_{v\in B}\,\int_{u\in\R^d}\,\phi_A(u-v)\mu(\dd u,\dd v)\,.
\end{equation}
As the reader may check,
\begin{equation}
{\Pr}_A(\mu * \gd_{(x,x)}) = ({\Pr}_A \mu) * \gd_x, \qquad \ x\in \bbR^d,
\end{equation}
so ${\Pr}_A$ can be naturally extended to a map $\tilde \Pr_A\colon \tilde\cX^{(2)}\mapsto \tilde\cX$. We then first observe that:
\begin{proposition}\label{pr:cont-smooth-proj}
The map $\tilde \Pr_A$ is continuous from $(\tilde \cX^{(2)}, \bfD^{(2)})$ to $(\tilde \cX,\bfD)$.
\end{proposition}
\begin{proof}[Proof of Proposition~\ref{pr:cont-smooth-proj}]
Given $\xi\in\tilde\cX^{(2)}$ and $f\in \cF_k$ (in particular $f$ is bounded), we have that
\begin{equation}
\sum_{\alpha\in\xi}\int f(v_1,\ldots, v_k)\prod_{i=1}^k({\Pr}_A\alpha)(\dd v_i)
= 	\sum_{\alpha\in\xi}\int f(v_1,\ldots, v_k)\prod_{i=1}^k
\phi_A(u_i-v_i)\alpha(\dd u_i, \dd v_i)\,.
\end{equation}
To conclude it only remains to observe that 
\begin{equation}
\label{eq:vanish0}
V_A\colon (u_1, v_1, \ldots, u_k, v_k) \mapsto f(v_1, \ldots, v_k)\prod_{i=1}^k \phi_A(u_i-v_i)\in \cF_k^{(2)}\,.
\end{equation}
Indeed, $V_A$ is clearly continuous and it is also vanishing: if we write, for convenience,
\beq
(u_1, v_1, \ldots, u_k, v_k) = (x_1, \ldots, x_{2k}),
\eeq
and assume that $\max\{|x_i - x_j| \colon i\neq j\}\ge M$, then one of the four following cases indeed occurs: 
\begin{enumerate}
\item $|u_i-v_i|\ge M$ for some index $i$;
\item $|v_i - v_j|\ge M$ for some index pair $i\neq j$;
\item $|u_i - v_j|\ge M$ for some index pair $i\neq j$;
\item $|u_i - u_j|\ge M$ for some index pair $i\neq j$.
\end{enumerate}
In the first case,  we write
\beq
V_A(u_1, v_1, \ldots, u_k, v_k) \le \|f\|_{\infty} \sup\{ \phi_A(x)\colon |x|\ge M\},
\eeq
which is eventually zero when $M$ is large enough, since $\phi_A$ is compactly supported. In the second case, we write
\beq \label{eq:vanish1}
V_A(u_1, v_1, \ldots, u_k, v_k) \le \sup\Big\{f(v_1, \ldots, v_k)\colon \max_{i\neq j} |v_i-v_j|\ge M\Big\},
\eeq
which decreases to zero as $M\to \infty$, since $f$ itself is vanishing. In the third case, we use the triangular inequality $|u_i - v_j| \le |u_i - v_i| + |v_i - v_j|$ so that we are back to one of the first two cases with $M/2$ instead of $M$. In the fourth case, we use the triangular inequality $|u_i - u_j| \le |u_i - v_j| + |v_j - u_j|$ so that we are back to either the first or third case with $M/2$ instead of $M$. Hence, the proof is complete.
\end{proof}
To continue we further define the diagonal with width $A$ in $\R^d\times\R^d$ via
\beq
\cD_A:=\{(u,v)\in\R^d\times\R^d:\, |u-v|\leq A\}\,,
\eeq
along with
\beq
\cM(A,\delta):=\{\xi\in\tilde\cX^{(2)}\,:\, \xi(\cD_A)\geq 1-\delta\}\,,
\eeq
where
\beq
\xi(\cD_A)=\sum_{\alpha\in \xi}\alpha(\cD_A)\,.
\eeq
The latter is well defined, i.e., it does not depend on the choice of the orbits. In Lemmas~\ref{lem:MAclosed} and~\ref{lem:MAcontained} below, we make two observations on the set $\cM(A,\delta)$ that shall prove useful in the sequel.
\begin{lemma}\label{lem:MAclosed}
	$\cM(A,\delta)$ is a closed subset of $\tilde\cX^{(2)}$.
\end{lemma}
\begin{proof}[Proof of Lemma~\ref{lem:MAclosed}]
	Let $(\xi_m)$ be a sequence in $\cM(A,\delta)$ converging with respect to the metric ${\bf D}_2$ to some element $\xi$. We need to show that $\xi\in\cM(A,\delta)$.
	To that end we define a collection of smooth functions $\phi_{A_1, A_2}\colon [0,\infty)\mapsto [0,1]$, where $A< A_1 < A_2$, with the following properties
	\begin{equation}
	\begin{aligned} \label{eq:def-phi-M}
	\phi_{A_1,A_2}(x)=\begin{cases}
	1\,, & A_1 < x < A_2\,,\\
	0\,, &x < \tfrac12(A+A_1) \text{ or } x> 2A_2\,,
	\end{cases}
	\end{aligned}
	\end{equation}
	and such that $\phi_{A_1,A_2}$ monotonically increases to the indicator of $(A,\infty)$ (pointwise) as $A_1$ and $A_2$ tend to $A$ and $+\infty$ respectively.
	Given any sub-probability measure $\alpha\in\cM_{\leq 1}(\R^d\times\R^d)$ we then see that by the monotone convergence theorem
	\beq
	\alpha (\cD_A^c)=\sup_{A <  A_1 < A_2}\int \phi_{A_1, A_2}(|u-v|)\alpha(\dd u,\dd v).
	\eeq
	We now introduce another collection of smooth functions $\varphi_B\colon [0,\infty) \to [0,1]$ such that $\varphi_B(x)= 1$ if $x\le B$ and $\varphi_B(x)= 0$ if $x>2B$, and such that $\varphi_B$ monotonically increases (pointwise) to the constant function equal to one as $B$ tends to infinity. Again, by the monotone convergence theorem:
	\beq
	\alpha (\cD_A^c)=\suptwo{A <  A_1 < A_2}{B>0} I(\ga; A_1, A_2, B).
	\eeq
	where
	\beq
	I(\ga; A_1, A_2, B) =\int \phi_{A_1, A_2}(|u_1-v_1|)
	\varphi_B(|u_1-u_2|)
	\varphi_B(|v_1-v_2|)
	\alpha(\dd u_1,\dd v_1)\alpha(\dd u_2,\dd v_2).
	\eeq
	We finally obtain:
	\beq
	\ba
	\sum_{\ga \in \xi} \alpha (\cD_A^c) &= \suptwo{A <  A_1 < A_2}{B>0} 
	\sum_{\ga \in \xi} I(\ga; A_1, A_2, B)\\
	& = \suptwo{A <  A_1 < A_2}{B>0} \lim_{m\to \infty}
	\sum_{\ga \in \xi_m} I(\ga; A_1, A_2, B)\\
	&\le \liminf_{m\to \infty}
	\sum_{\ga \in \xi_m} \alpha (\cD_A^c) \le \gd.
	\ea
	\eeq
	On the second line we used the convergence of $\xi_m$ to $\xi$ along with the fact that the function
	\beq
	(u_1,v_1,u_2,v_2)\in (\bbR^d)^4  \mapsto \phi_{A_1, A_2}(|u_1-v_1|)
	\varphi_B(|u_1-u_2|)
	\varphi_B(|v_1-v_2|)
	\eeq
	is smooth, translation invariant and vanishing, since $\phi_{A_1, A_2}$ and $\varphi_B$ are compactly supported. This completes the proof.
\end{proof}
\begin{lemma}\label{lem:MAcontained}
	For every $A, \gd >0$,
	\beq
	\cM(A,\delta)\subseteq \{\xi\in\tilde\cX^{(2)}\,:\,{\bf D}(\tilde \Pr(\xi), \tilde \Pr_A(\xi))\leq 2\delta \}\,.
	\eeq
\end{lemma}
\begin{proof}[Proof of Lemma~\ref{lem:MAcontained}] 
We first estimate the distance between $\tilde \Pr_A$ and $\tilde \Pr$. Given any function $f\in\cF_k$ and any sub-probability measure $\mu\in\cM_{1}(\R^d\times\R^d)$,
		\begin{equation}
		\begin{aligned}
		\int &f(v_1,\ldots, v_k)\Big[\prod_{i=1}^k({\Pr}\mu)(\dd v_i)-\prod_{i=1}^k ({\Pr}_A\mu)(\dd v_i)\Big]\\
		&=\sum_{j=1}^k\int f(v_1,\ldots, v_k)\big[({\Pr}\mu-{\Pr}_A\mu)(\dd v_j)\big]\prod_{i<j}({\Pr}\mu)(\dd v_i)\prod_{i>j}({\Pr}_A\mu)(\dd v_i)\\
		&= \sum_{j=1}^k\int f(v_1,\ldots, v_k)\big[(1-\phi_A)(u_j-v_j)\mu(\dd u_j,\dd v_j)\big]\prod_{i<j}({\Pr}\mu)(\dd v_i)\prod_{i>j}({\Pr}_A\mu)(\dd v_i)\,.
		\end{aligned}
		\end{equation}
		By~\eqref{eq:def-phi-n}, the above can be  bounded from above (in absolute value) by
		\begin{equation}
		k\|f\|_\infty \times  \mu(\{(u,v)\in \bbR^d\times \bbR^d\,:\, |u-v|\ge A\})\,.
		\end{equation}
		From the definition of the metric ${\bf D}$ in~\eqref{eq:metricD} and~\cite[Section 3, Eq. (17)]{MV2016}, we may deduce that
		\begin{equation}\label{eq:off_diag_est}
		{\bf D}(\tilde \Pr (\xi), \tilde \Pr_A(\xi)) \le 2\times \xi(\cD_A^c).
		\end{equation}
\end{proof}

Our final key observation before we state the main result of this section is the following:
\begin{proposition}\label{ass:diag}
	The sequence $(\tilde L_n^{(2)})$ is diagonally exponentially tight, i.e. for all $\delta>0$
	\beq
	\lim_{A\to\infty}\limsup_{n\to\infty} \frac1n \log \bP(\tilde L_n^{(2)}\notin \cM(A, \delta))=-\infty\,.
	\eeq	
\end{proposition}
\begin{proof}[Proof of Proposition~\ref{ass:diag}]
Fix $\delta>0$ and note that
\beq
\bP(\tilde L_n^{(2)}\notin \cM(A, \delta))=\bP(\mathrm{Bin}(n, \bP(|X_1|\geq A))>n\delta)\,,
\eeq
where we recall that $X_1$ is distributed as a single step of the underlying Markov chain. Let us abbreviate $p_A:= \bP(|X_1|\geq A)$.
 Using Chernov's bound we may write for all $\lambda>0$
\begin{equation}
	\begin{aligned}
		\bP(\mathrm{Bin}(n, \bP(|X_1|\geq A))>n\delta)
		&\leq \exp\Big(-n\Big[\lambda\delta -\log\Big(1+p_A(e^\lambda -1)\Big)\Big]\Big)\\
		&\leq  \exp\Big(-n\Big[\lambda\delta -p_A(e^\lambda -1)\Big]\Big)\,.
	\end{aligned}	
	\end{equation}
Choosing $\lambda =\log(p_A^{-1})$ the right-hand side above becomes
\begin{equation}
\exp\Big(-n\Big[\log(p_A^{-1})\delta -1+p_A\Big]\Big)\,.
\end{equation}
To conclude it suffices to note that $\lim_{A\to\infty}\log(p_A^{-1}) =+\infty$.
\end{proof}

We can now formulate the following result:
\begin{theorem}
	\label{thm:LDPempirical}
Under the assumptions (1) to (3) in Section~\ref{sec:LDP}, the sequence $(\tilde L_n)$ satisfies a Large Deviation principle in $\tilde\cX$ with rate $n$ and rate function
\beq\label{eq:def-tildeJ}
\tilde J(\xi)=\sum_{\alpha\in\xi} J(\alpha)\,,
\eeq
where
\beq
J(\alpha)=  \inf\{J^{(2)}(\mu)\colon \Pr(\mu)= \ga\}\,.
\eeq
\end{theorem}
Before we proceed with the proof, let us remark that the rate function $\tilde J$ in~\eqref{eq:def-tildeJ} coincides with what we would get from a blind application of the contraction principle, that is~\eqref{eq:tildeJ} in the following lemma:
\begin{lemma}\label{lem:tildeJ} Let $\tilde J$ be defined as in~\eqref{eq:def-tildeJ} and $\xi\in\tilde \cX$. Then,
\beq\label{eq:tildeJ}
\tilde J(\xi) = \inf\{\tilde J^{(2)}(\gt) \colon \tilde \Pr(\gt)= \xi\}.
\eeq
Moreover, $\tilde J$ is a lower semi-continuous function.
\end{lemma}

\begin{proof}[Proof of Lemma~\ref{lem:tildeJ}]
Write $\xi=\{\ga_i:\, i\in I\}$. Note that for every $\theta$ such that $\tilde \Pr(\theta) = \xi$,
\beq
\sum_{i\in I} J(\ga_i) \le \tilde J^{(2)}(\theta).
\eeq
Indeed, writing $\theta = \{\nu_i \colon i \in I\}$ we have $\tilde \Pr(\nu_i) = \ga_i$ for every $i\in I$, hence $\tilde J^{(2)}(\theta) = \sum_{i\in I} J^{(2)}(\nu_i) \ge \sum_{i\in I} J(\ga_i)$. Therefore, $\tilde J(\xi) \ge \sum_{i\in I} J(\ga_i)$.
We now proceed with the upper bound. We can assume that the index set $I$ is a subset of the natural numbers or coincides with it. Fix $\gep>0$. For all $i\in I$, there exists $\nu_i^{(\gep)}$ such that $\Pr(\nu_i^{(\gep)})= \ga_i$ and $|J^{(2)}(\nu_i^{(\gep)})-J(\ga_i)|\leq \gep 2^{-i}$. Define $\theta^{(\gep)} = \{\nu_i^{(\gep)}, i\in I\}$. Then $\tilde \Pr \theta^{(\gep)} = \xi$, from which it follows that:
\beq
\tilde J(\xi) \le \tilde J^{(2)}(\theta^{(\gep)}) = \sum_{i\in I} J^{(2)}(\nu_i^{(\gep)}) \leq \sum_{i\in I} J(\ga_i) + \gep\,.
\eeq

\par To show that $\tilde J$ is lower semi-continuous, we first remark that the function $J$ appearing in~\eqref{eq:def-tildeJ} may be written as
\beq
J(\ga) = \sup_v \int \log \Big(\frac{v}{\pi v}\Big) \dd \ga,
\eeq
where the supremum runs over all bounded and Borel-measurable (or continuous and compactly supported) functions $v\colon \bbR^d \to [1, +\infty)$. This standard fact may be inferred, for instance, from~\cite[Lemma 2.1 of the first paper]{DV75} and~\cite[Theorem 2.1]{DV76}. From this variational representation, one can
mimick the proof of~Proposition~\ref{pr:lsc}, that is itself inspired by~\cite[Lemma 4.2]{MV2016}.
\end{proof}
\begin{proof}[Proof of Theorem~\ref{thm:LDPempirical}]We start with the proof of the upper bound.
Let $F\subseteq \tilde\cX$ be a closed set. We can then estimate
\beq
\ba
\bP(\tilde L_n\in F)&\leq 
\bP(\tilde L_n \in F, \tilde L_n^{(2)}\in \cM(A,\delta))
+ \bP(\tilde L_n^{(2)}\notin \cM(A,\delta))\\
&\leq \bP(\tilde L_n^{(2)}\in\tilde{\Pr}_A^{-1}(F_{2\delta}), \tilde L_n^{(2)}\in \cM(A,\delta)) + \bP(\tilde L_n^{(2)}\notin \cM(A,\delta))\,.
\ea
\eeq
Here, for any $a>0$ we defined the closed set
\beq
F_a=\{\xi\in \tilde\cX\,:\, {\bf D}(\xi, F)\leq a\}\,,
\eeq
and the second inequality above follows from Lemma~\ref{lem:MAcontained}. Since $F_{2\delta}$ is closed and $\tilde\Pr_A$ is continuous it follows that $\tilde{\Pr}_A^{-1}(F_{2\delta})$ is closed. 
For every $A,\gd>0$, we denote:
\beq
	-C(A,\delta):=\limsup_{n\to\infty} \frac1n \log \bP(\tilde L_n^{(2)}\notin \cM(A, \delta))\,.
\eeq
Using Lemma~\ref{lem:MAclosed} and the Large Deviation principle for the pair empirical measure we can conclude that
\beq
\ba
\limsup_{n\to\infty} \frac1n \log \bP(\tilde L_n\in F)
&\leq -\Big(C(A,\delta)\wedge \inf_{\xi\in\cM(A,\delta)\cap \tilde\Pr_A^{-1}(F_{2\delta})}\tilde J^{(2)}(\xi)\Big)\\
&\leq -\Big(C(A,\delta)\wedge \inf_{\xi\in\cM(A,\delta)\cap \tilde{\Pr}^{-1}(F_{4\delta})}\tilde J^{(2)}(\xi)\Big)\\
&\leq -\Big(C(A,\delta)\wedge \inf_{\xi\in \tilde{\Pr}^{-1}(F_{4\delta})}\tilde J^{(2)}(\xi)\Big)
\ea
\eeq
where the second-to-last inequality follows from Lemma~\ref{lem:MAcontained}. Note that by Proposition~\ref{ass:diag} the constant $C(A,\delta)$ tends to infinity when $A$ tends to infinity. Therefore, we see that for any $\delta>0$
\beq
\limsup_{n\to\infty} \frac1n \log \bP(\tilde L_n\in F)
\leq -\inf_{\xi\in F_{\delta}}\tilde J(\xi)\,,
\eeq
where
\beq
\label{eq:def_tildeJ}
\tilde J(\xi) := \inf\{\tilde J^{(2)}(\theta) \colon \tilde\Pr(\theta) = \xi\}.
\eeq
The fact that
\beq
\tilde J(\xi) = \sum_{\ga\in \xi} J(\ga),
\eeq
and that $\tilde J$ is lower semi-continuous are consequences of Lemma~\ref{lem:tildeJ} above.
To finish the derivation of the Large Deviation upper bound we need to send $\delta$ to $0$. Since $F_\delta$ is a closed set of a compact space it is compact itself. Since $\tilde J$ is lower semi-continuous there exists $\xi_\delta\in F_\delta$ minimizing $\tilde J$ over $F_\delta$. By compactness of $\tilde\cX$ we can extract a subsequence along which $\xi_\delta$ converges as $\delta\to 0$. For ease of notation we denote this subsequence by $\xi_\delta$ and its limit by $\xi_0$. By the lower semi-continuity of $\tilde J$ we have that
\beq
\tilde J(\xi_0)\leq \liminf_{\delta \to 0}\tilde J(\xi_\delta)\,.
\eeq
Since $\xi\mapsto {\bf D}(\xi, F)$ is continuous it follows that $\xi_0\in F$. Therefore,
\beq
\inf_{\xi \in F}\tilde J(\xi)\leq \tilde J(\xi_0)\leq \liminf_{\delta \to 0}\tilde J(\xi_\delta)\,,
\eeq
which allows to conclude the upper bound.\\ 
\par We now come to the proof of the lower bound. Let $U$ be an open set in $\tilde\cX$ w.r.t the metric $\bfD$. We first show that
\beq\label{eq:lowermain}
\liminf_{n\to\infty} \frac1n \log \bP(\tilde L_n \in U) \ge - \liminf_{A\to \infty} \inf_{\tilde\Pr_A^{-1}(B(\xi,\gd))} \tilde J^{(2)}\,,
\eeq
where $\xi \in U$ and $\gd$ is small enough such that $B(\xi,\gd) \subseteq U$. To see that, we can write using Lemma~\ref{lem:MAcontained},
\beq
\ba
\bP(\tilde L_n \in U)
&\geq \bP(\tilde L_n \in B(\xi,\delta), \tilde L_n^{(2)}\in \cM(A,\delta/4))\\
&\geq \bP(\tilde\Pr_A(\tilde L_n^{(2)}) \in B(\xi,\delta/2),\tilde L_n^{(2)} \in \cM(A,\delta/4))\\
&\geq \bP(\tilde\Pr_A(\tilde L_n^{(2)}) \in B(\xi,\delta/2))- \bP(\tilde L_n^{(2)}\notin\cM(A,\delta/4))\,.
\ea
\eeq
Applying the Large Deviation Principle for $\tilde L_n^{(2)}$ and Proposition~\ref{ass:diag} allows to conclude~\eqref{eq:lowermain}.
	To continue note that by Equation~\eqref{eq:off_diag_est} for any $\xi^{(2)}\in\tilde\cX^{(2)}$ one has that ${\bf D}(\tilde\Pr(\xi^{(2)}), \tilde\Pr_A(\xi^{(2)}))\leq 2 \xi^{(2)}(\cD_A^c)$, which tends to zero as $A\to\infty$. Indeed, note that
	\beq
	\xi^{(2)}(\cD_A^c) =\sum_{\ga\in\xi^{(2)}}\alpha(\cD_A^c)\,,
	\eeq
	and each term in the sum tends to zero as $A\to\infty$. Since moreover, $\alpha(\cD_A^c)\leq \alpha(\R^d\times\R^d)$ and the latter is summable over all the elements in $\xi$, we can conclude using the dominated convergence theorem.
	 Hence, given any $\xi^{(2)}\in\tilde\cX^{(2)}$ such that $\tilde\Pr(\xi^{(2)})=\xi$ there exists $A_0>0$ such that for all $A\geq A_0$ one has that $\tilde\Pr_A(\xi^{(2)})\in B(\xi,\delta)$. We define
	\beq
	\cC_A=\{\xi^{(2)}\in\tilde\cX^{(2)}:\, \tilde\Pr(\xi^{(2)})=\xi, \tilde{\Pr}_A(\xi^{(2)})\in B(\xi,\delta)\}\,.
	\eeq
	It follows directly from the definition of $\cC_A$ that
	\beq
	\inf_{\tilde\Pr_A^{-1}(B(\xi,\gd))} \tilde J^{(2)}\leq \inf_{\cC_A}  \tilde J^{(2)}\,.
	\eeq
	Our goal is to prove that 
	\beq\label{eq:toshow}
	\limsup_{A\to\infty} \inf_{\cC_A}  \tilde J^{(2)}\leq \inf_{\Pr^{-1}(\{\xi\})}\tilde J^{(2)}\,,
	\eeq
	which would yield the result.
	Let $\gep>0$, and fix $\xi_\gep^{(2)}\in\tilde\cX^{(2)}$ such that $\Pr(\xi_\gep^{(2)})=\xi$ and such that additionally
	\beq
	\inf_{\Pr^{-1}(\{\xi\})}\tilde J^{(2)}\geq \tilde J^{(2)}(\xi_\gep^{(2)})-\gep\,.
	\eeq
	By the above, there exists $A_0>0$ such that for all $A\geq A_0$ one has that $\xi_\gep^{(2)}\in\cC_A$. This allows us to conclude that
	\beq
	\inf_{\Pr^{-1}(\{\xi\})}\tilde J^{(2)}\geq \tilde J^{(2)}(\xi_\gep^{(2)})-\gep \geq \inf_{\cC_A}\tilde J^{(2)} -\gep\,
	\eeq
	for all $A\geq A_0$. Sending $A$ to infinity and $\gep$ to zero allows us to deduce~\eqref{eq:toshow}. Hence, given $\xi\in\tilde\cX$ and any open set $U$ containing $\xi$ we can conclude that
	\beq
	\liminf_{n\to\infty} \frac1n \log \bP(\tilde L_n \in U) \ge -	\inf_{\tilde\Pr^{-1}(\{\xi\})}\tilde J^{(2)}
	=- \tilde J(\xi)\,,
	\eeq
	and the result follows.
	To finish the proof of Theorem~\ref{thm:LDPempirical} we need to establish the lower semi-continuity of the rate function. This however is the content of Lemma~\ref{lem:tildeJ}. Hence, we can conclude.
\end{proof}
\subsection{Additional observations}
Let us come back to the counter-example provided at the beginning of this section and formulate a few observations that may be of independent interest. Looking back at~\eqref{eq:c-ex}, we remark that the left-hand side therein is in some sense ``larger" (i.e.\ contains more mass) than the right-hand side. We formalize this sort of ``lower semi-continuity'' with the following:

\begin{proposition}\label{pr:proj-lsc}
	Assume that the sequence $(\xi_n)$ in $\tilde \cX^{(2)}$ converges to $\xi$ w.r.t to the metric $\Dtwo$. Then, for every non-negative function $f\in \cF_2$,
	\beq
	\liminf_{n\to \infty} \int f(v_1, v_2) \tilde\Pr \xi_n (\dd v_1) \tilde\Pr \xi_n (\dd v_2) \ge 
	\int f(v_1, v_2) \tilde\Pr \xi (\dd v_1) \tilde\Pr \xi (\dd v_2).
	\eeq
\end{proposition}

\begin{proof}[Proof of Proposition~\ref{pr:proj-lsc}]
	We use the smooth truncation functions from~\eqref{eq:def-phi-n}. 
	Then, for all $A>0$,
	\beq
	\ba
	\int f(v_1, v_2) \tilde\Pr \xi_n (\dd v_1) \tilde\Pr \xi_n (\dd v_2) &=
	\int f(v_1, v_2) \xi_n (\dd u_1, \dd v_1)  \xi_n (\dd u_2, \dd v_2)\\
	&\ge \int f_A(u_1,v_1,u_2, v_2) \xi_n (\dd u_1, v_1)\xi_n (\dd u_2, \dd v_2),
	\ea
	\eeq
	where
	\beq
	f_A(u_1,v_1,u_2, v_2) := f(v_1,v_2)\phi_A(u_1-u_2).
	\eeq
	Repeating the arguments between~\eqref{eq:vanish0} and~\eqref{eq:vanish1}, one can show that $f_A \in \cF_2^{(2)}$. Therefore,
	\beq
	\liminf_{n\to \infty} \int f(v_1, v_2) \tilde\Pr \xi_n (\dd v_1) \tilde\Pr \xi_n (\dd v_2) \ge 
	\int f_A(u_1,v_1,u_2, v_2) \xi (\dd u_1, v_1)\xi (\dd u_2, \dd v_2).
	\eeq
	Letting $A\to\infty$ in the right-hand side above and using the monotone convergence theorem completes the proof.
\end{proof}
The proof above indicates that the concept of {\it diagonal tightness} introduced above should be sufficient to retrieve continuity of the projection map. This is the purpose of the below proposition, in which we use the following notation: if $\xi = \{\tilde \ga_i\}_{i\in I}\in \tilde \cX^{(2)}$ and if $A$ is a {\it diagonally shift invariant} Borel subset of $\bbR^d$ (meaning that $A+(x,x)=A$ for every $x\in\bbR^d$) then
\beq
\xi(A) := \sum_{i\in I} \ga_i(A),
\eeq
where the choice of $\ga_i$ in the orbit $\tilde \ga_i$ is irrelevant. We may now state:
\begin{proposition}\label{pr:proj-diag}
	Assume that the sequence $(\xi_n)$ in $\tilde \cX^{(2)}$ converges to $\xi$ w.r.t to the metric $\Dtwo$. Assume in addition that this sequence is \emph{diagonally tight}, meaning that
	\beq
	\lim_{M\to \infty}\sup_{n\in\bbN} \xi_n(\{(u,v)\colon |u-v| >M\})=0.
	\eeq
	Then, the sequence $\tilde \Pr(\xi_n)$ converges to $\tilde \Pr(\xi)$ w.r.t. to the metric ${\bf D}$.
\end{proposition}

Let us remark that any \emph{finite} collection of elements in $\tilde \cX^{(2)}$ is diagonally tight, since any probability measure on the Polish space $\bbR^d$ is tight and the total mass of an element in $\tilde \cX^{(2)}$ is bounded by one.

\begin{proof}[Proof of Proposition~\ref{pr:proj-diag}]
	We write
	\beq
	{\bfD}(\tilde \Pr(\xi_n),\tilde \Pr(\xi))\leq {\bfD}(\tilde \Pr(\xi_n), \tilde \Pr_A(\xi_n)) + {\bfD}(\tilde \Pr_A(\xi_n), \tilde \Pr_A(\xi)) + {\bfD}(\tilde \Pr_A(\xi), \tilde \Pr(\xi))\,.
	\eeq
	The first term can be estimated from above by $2\times \xi_n(\cD_A^c)$ by~\eqref{eq:off_diag_est}, which tends to zero as $A$ tends to infinity, uniformly in $n$, from our assumption. The second term tends to zero as $n\to\infty$ for every $A$ by the continuity of $\tilde \Pr_A$ proven in Proposition~\ref{pr:cont-smooth-proj}. Finally, the last term above tends to zero as $A$ tends to infinity, since by~\eqref{eq:off_diag_est} it can be estimated from above by $2\times \xi(\cD_A^c)$. This concludes the proof.
\end{proof}

\appendix
\section{On the Mukherjee-Varadhan topology}\label{sec:a}
This section collects the most important ingredients about the topology introduced by Mukherjee and Varadhan in~\cite[Section 3]{MV2016}.
Similarly to Section~\ref{sec:topology} we write $\cM_1= \cM_1(\bbR^d)$ for the space of probability measures on $\bbR^d$ and $\cM_{\le1} = \cM_{\le1}(\bbR^d)$ for the space of sub-probability measures on $\bbR^d$. For any $\alpha\in \cM_{\le 1}$ we define its orbit $\tilde\alpha=\{\alpha*\delta_x:\, x\in\R^d\}$, and we identify two measures if they have the same orbit. This introduces an equivalence relation on  $\cM_{\le 1}$ of which $\tilde \cM_{\le 1}$ denotes the corresponding quotient space. For $k\ge 2$, define $\cF_k$ as the space of continuous functions $f\colon (\bbR^d)^k \mapsto \bbR$ that are {\it translation invariant}, i.e.
\beq
f(u_1+x, \ldots, u_k +x) = f(u_1,\ldots, u_k), \quad \forall x,u_1,\ldots, u_k \in \bbR^d,
\eeq
and {\it vanishing at infinity}, in the sense that
\beq
\lim_{\max_{i\neq j} |u_i-u_j|\to \infty} f(u_1, \ldots, u_k) = 0.
\eeq
For $k\ge 2$, $f\in \cF_k$ and $\ga \in \cM_{\le 1}$, write
\beq
\label{eq:defLambda0}
\gL(f, \ga) := \int f(u_1,\ldots, u_{k}) \prod_{1\le i\le k} \ga(\dd u_i),
\eeq
which only depends on $\tilde \ga$. 
Define
\beq
\cF := \bigcup_{k\ge 2} \cF_k,
\eeq
for which there exists a countable dense set (under the uniform metric) denoted by
\beq
\{f_r(u_1,\ldots, u_{k_r}),\, r\in \bbN\}.
\eeq
One can then define the following set of empty, finite, or countably infinite collections of sub-probability measure orbits:
\beq
\tilde\cX := \Big\{
\xi = \{\tilde \ga_i\}_{i\in I} \colon \tilde \ga_i \in \tilde \cM_{\le 1},\ \sum_{i\in I} \ga_i(\bbR^d) \le 1\Big\}\,,
\eeq
see Equation (16) in~\cite{MV2016}.
For every $\xi_1, \xi_2\in \tilde \cX$, one can then define the following metric on $\tilde\cX$:
\beq\label{eq:metricD}
\mathbf{D}(\xi_1, \xi_2) := \sum_{r\ge 1} \frac{1}{2^r} \frac{1}{1+\|f_r\|_{\infty}}
\Big|%
\sum_{\tilde\ga\in\xi_1} \gL(f_r, \ga)- \sum_{\tilde\ga\in\xi_2} \gL(f_r, \ga)
\Big|.%
\eeq
It was then shown in~\cite{MV2016} that the space $\tilde\cX$ equipped with $\mathbf{D}$ is a compact metric space and that $\tilde \cM_1$ is dense in $\tilde\cX$. 
\section*{Acknowledgements}
D.E.~was supported by the National Council for Scientific and Technological Development - CNPq via a Bolsa de Produtividade 303520/2019-1 and 303348/2022-4 and via a Universal Grant (Grant Number 406001/2021-9). D.E. moreover acknowledges support by the Serrapilheira Institute (Grant Number Serra-R-2011-37582). Finally, D.E. was partially
supported by FAPESB (EDITAL FAPESB Nº 012/2022 - UNIVERSAL - NºAPP0044/2023). FAPESB is the Bahia Research Foundation. J.P. is supported by the ANR-22-CE40-0012 LOCAL. Both authors would like to thank an anonymous referee for helpful remarks on a previous version of this paper.\\

\bibliographystyle{abbrv}
\bibliography{SwissCheese}

\end{document}